\documentclass[sn-mathphys-num]{sn-jnl}

\usepackage{graphicx}%
\usepackage{multirow}%
\usepackage{amsmath,amsthm,amssymb,amsfonts}
\usepackage{mathrsfs}%
\usepackage[title]{appendix}%
\usepackage{xcolor}%
\usepackage{textcomp}%
\usepackage{manyfoot}%
\usepackage{booktabs}%
\usepackage{algorithm}%
\usepackage{algorithmicx}%
\usepackage{algpseudocode}%
\usepackage{listings}%
\usepackage{dsfont}

\newtheorem{theorem}{Theorem}
\newtheorem{proposition}[theorem]{Proposition}%
\newtheorem{problem}{Problem}
\newtheorem{remark}{Remark}%
\newtheorem{assumptions}{Assumption}
\newtheorem{lemma}{Lemma}
\newtheorem{corollary}{Corollary}
\newtheorem{definition}{Definition}%

\usepackage{caption}
\usepackage{subcaption}
\usepackage{booktabs}

\newcommand{\Mc}{\mathcal{M}}
\newcommand{\M}{\mathcal{M}}
\newcommand{\N}{\mathbb{N}}
\newcommand{\E}{\mathbb{E}}

\newcommand{\Xc}{\mathcal{X}}
\newcommand{\X}{\mathcal{X}}
\newcommand{\F}{\mathcal{F}}
\newcommand{\G}{\mathcal{G}}
\newcommand{\supp}{\text{spt}}
\newcommand{\dist}{\text{dist}}

\newcommand{\U}{\mathcal{U}}

\newcommand{\veps}{\varepsilon}
\newcommand{\R}{\mathbb{R}}
\newcommand{\Ric}{{\rm Ric}}

\newcommand{\argmin}{{\rm argmin}}
\newcommand{\vol}{{\rm vol}}
\renewcommand{\k}{\kappa}

\newcommand{\nc}{\normalcolor}

\newcommand\restr[2]{{
  \left.\kern-\nulldelimiterspace 
  #1 
  \vphantom{\big|} 
  \right|_{#2} 
  }}

\begin{document}

\title[Continuum Limits of Ollivier's Ricci Curvature on data clouds]{Continuum Limits of Ollivier's Ricci Curvature on data clouds: pointwise consistency and global lower bounds}


\author[1]{\fnm{Nicol\'as} \sur{Garc\'ia Trillos}}

\author[2]{\fnm{Melanie} \sur{Weber}}

\affil[1]{\orgdiv{Department of Statistics}, \orgname{University of Wisconsin-Madison}}

\affil[2]{\orgname{Harvard University}}


\abstract{\noindent Let $\Mc \subseteq \mathbb{R}^d$ denote a low-dimensional manifold and let $\Xc= \{ x_1, \dots, x_n \}$ be a collection of points uniformly sampled from $\Mc$. We study the relationship between the curvature of a random geometric graph built from $\Xc$ and the curvature of the manifold $\Mc$ via continuum limits of Ollivier's discrete Ricci curvature.  We prove pointwise, non-asymptotic consistency results and also show that if $\Mc$ has Ricci curvature bounded from below by a positive constant, then the random geometric graph will inherit this global structural property with high probability. We discuss applications of the global discrete curvature bounds to contraction properties of heat kernels on graphs, as well as implications for manifold learning from data clouds.  In particular, we show that our consistency results allow for estimating the intrinsic curvature of a manifold by first estimating concrete extrinsic quantities.}

\keywords{Ollivier-Ricci curvature, curvature lower bounds, random geometric graphs, discrete-to-continuum consistency, geometric inference, graph-based learning, manifold learning}


\pacs[MSC Classification]{52C99, 53A35, 62G05, 68U05}

\maketitle
\section{Introduction}
The problem of identifying geometric structure in data is central to Machine Learning and Data Science.  A frequently encountered structure is \emph{low-dimensionality}, where high-dimensional data is assumed to lie on or near a low-dimensional manifold (\emph{manifold hypothesis}).  Let $\Xc \subset \mathbb{R}^d$ denote such a data set of size $n$ and $\Mc \subseteq \R^d$ a low-dimensional manifold from which $\Xc$ was sampled.  Given $\Xc$,  but without prior knowledge of $\Mc$, what can we say about the \textit{intrinsic} geometry of $\Mc$? In particular, what can we learn about intrinsic notions of curvature of $\Mc$ from $\Xc$? One of the central goals of this paper is to study this question through the lens of discrete Ricci curvature. Traditionally, curvature has been studied in continuous spaces such as Riemannian manifolds. Several different notions of curvature relate to the local and global geometric properties of manifolds. \emph{Ricci curvature} is a classical concept in Riemannian geometry, which, in particular, determines the volume growth of geodesic balls and which is closely connected to functional inequalities. In this paper, we study a \emph{discrete} notion of Ricci curvature originally defined by Ollivier~\citep{ollivier2009ricci} and its relationship to the classical Ricci curvature on $\M$. 
More specifically, we study the relationship between the curvature of a suitable \emph{random geometric graph} (short: \emph{RGG}) built from $\Xc$ and the curvature of the manifold $\Mc$.  
A RGG $G$ is constructed from a sample $\Xc$ by connecting points with a \textit{distance} of at most $\veps$ with an edge; as we will discuss below, the choice of distance function plays an important role in our results.

In more concrete terms, we explore the following two questions:

\medskip

\begin{problem}\label{prob1}
Can we give \emph{non-asymptotic} error bounds for the \emph{pointwise estimation} of the curvature of $\Mc$ from that of $G$?
\end{problem}

\medskip 

\begin{problem}\label{prob2}
If the manifold $\M$ has Ricci curvature bounded from below by a given constant, will a RGG inherit this global structural property with high probability? What are some consequences of the resulting discrete curvature lower bounds? \nc
\end{problem}

\medskip

We will discuss both questions in two different settings.  In the first, which is of mostly theoretical value, we assume to have access to the pairwise \textit{geodesic} distances of points in $\Xc$, i.e.,  in $G$ two points are connected by an edge if they are within distance $\veps$ from each other along the manifold.  In the second setting,  we have no access to geodesic distances but we assume, instead, to have access to sufficiently accurate data-driven approximations thereof. In studying these two problems we will be able to provide theoretical insights into the relationship between discrete and continuous Ricci curvature and deliver consistent continuum limits of Ollivier's Ricci curvature on data clouds. Recall that a positive global lower bound on the Ricci curvature has several important implications for the manifold's geometry, including a bound on the diameter of complete manifolds (Bonnet-Myers~\citep{bonnet-myers}), as well as consequences for the coupling of random walks, which we will discuss below. We will explore some novel implications of having discrete curvature lower bounds on the behavior of graph Laplacians built over $G$. For example, the results presented in section \ref{sec:LaplaciansandRegression} are novel in the literature of graph Laplacians and do not follow from existing discrete-to-continuum consistency results. 



\subsection{Outline}

In order to precisely state our main results, we first present some background material. In particular, in section \ref{sec:DiffGeometry} we recall important concepts from differential geometry, including notions of Ricci curvature, parallel transport, and the second fundamental form of an embedded manifold; all these notions will be used in the sequel. Then, in section \ref{sec:Ollivier}, we discuss the notion of Ollivier's Ricci curvature for triplets $(\U, d, \mathbf{m})$ consisting of a metric space $(\U,d)$ and a Markov kernel $\mathbf{m}$ over $\U$; we will discuss a special setting where $\U$ is a Riemannian manifold $\M$ and discuss the connection between the induced Ollivier-Ricci curvature and the classical Ricci curvature discussed in section \ref{sec:DiffGeometry}. In section \ref{sec:DiscreteOllivier} we introduce the RGGs $G$ over $\X$ that we will study throughout the paper and define associated discrete Ollivier-Ricci curvatures up to the choice of a metric $d_G$ over $\X$. The metric $d_G$ will be explicitly defined in section \ref{sec:GeodesicDistancesX}, right after discussing the approximation of geodesic distances over $\M$ from data at the beginning of section \ref{sec:APproximationGeodesics}.  

In section \ref{sec:MainResults} we present our main theoretical results: in section \ref{sec:MainPointwiseState} our pointwise consistency results (Theorems \ref{thm:consistency-non-asymp} and \ref{thm:consistency-non-asymp-geo}), and in section \ref{sec:MainGlobal} our global lower bounds (Theorems \ref{thm:GlobalBounds1} and \ref{thm:GlobalBounds2}). In section \ref{sec:MainNumerical} we illustrate the recovery of Ricci curvature from data with a simple numerical experiment. 

In section \ref{sec:Literature} we discuss some related literature. 

In section \ref{sec:NonAsymptoticGuarantees} we present the proofs of our main results.

In section \ref{sec:Applications} we present a brief discussion of some applications of our main theoretical results. One such application is discussed in section \ref{sec:LaplaciansandRegression}, where we study the Lipschitz contractivity of graph heat kernels over data clouds sampled from manifolds with positive curvature. In section \ref{sec:ManifoldLearning} we discuss some of the avenues that our main results may open up in the field of manifold learning.

We wrap up the paper in section \ref{sec:Conclusions} with some conclusions and some discussion on directions for future research.



\section{Background and Notation}
\label{sec:background}
\subsection{Notions from Differential Geometry}
\label{sec:DiffGeometry}
We start by recalling some basic definitions and tools from differential geometry that we have collected from Chapters 0-4 and 6 in \citep{do1992riemannian}. This gives us the opportunity to introduce some notation that we use in the sequel. 

An $m$-dimensional \emph{manifold} $\M$ is a locally Euclidean space of dimension $m$ with a differentiable structure. We use $T_x \Mc$ to denote the tangent plane at the point $x \in \Mc$. Throughout the paper we will only consider smooth, connected, compact \emph{Riemannian} manifolds without boundary. These are manifolds endowed with a smoothly varying inner product structure $g=\{g_x\}_{x \in \M}$ defined over their tangent planes. The geodesic distance $d_\M$ between two points $x,y \in \M$ is defined according to
\[ d_\M(x,y) = \inf _{\gamma: [0,1] \rightarrow \M  } \int_{0}^1 
\sqrt{ g_{\gamma(t)} (\dot{\gamma}(t), \dot{\gamma}(t)) } dt,\]
where the inf ranges over all smooth paths connecting $x$ to $y$. Important notions in Riemannian geometry such as geodesic curves (in particular length minimizers) and curvature are defined in terms of \textit{connections} or \textit{covariant derivatives}. Informally, given a smooth curve $\gamma$ on $\M$, the covariant derivative $\nabla_{\dot{\gamma}}  $ is an operator, satisfying some linearity and Leibniz product rule properties, mapping vector fields along $\gamma$ into vector fields along $\gamma$. Among the multiple choices of connection that can be defined over a manifold, we will work with the Levi-Civita connection, which satisfies some additional compatibility conditions with the Riemannian structure of the manifold; see details in Chapter 2 in \citep{do1992riemannian}.

In general, a geodesic is a smooth path $\gamma:[0,1] \rightarrow \M$ satisfying $\nabla_{\dot \gamma} \dot{\gamma}=0$. It is possible to show that for every $x\in \M$ and every $v \in T_x\M$ there exists a unique geodesic $\gamma$ satisfying $\gamma(0)=x$ and $\dot{\gamma}(0)=v$. We may use this fact to define the \textit{exponential map} $ \exp_x: T_x \M \rightarrow \M$, which maps $v$ to $\gamma(1)$ for $\gamma$ the geodesic starting at $x$ in the direction $v$. It can be shown that there exists a real number $\iota_\M>0$, known as $\M$'s \textit{injectivity radius}, for which $\exp_{x} : B(0,\veps) \subseteq T_x\M \rightarrow B_\M(x,\veps)$ is a diffeomorphism for all $x \in \M$ and all $\veps < \iota_\M$; here and in the remainder, we use $B_\M(x,\veps)$ to denote the ball of radius $\veps$ around $x$ when $\M$ is endowed with $d_\M$ and $B(0,\veps)$ is the $m$-dimensional Euclidean ball of radius $\veps$.  The inverse of $\exp_x$, the \emph{logarithmic map}, will be denoted by 
$\log_x: B_\M(x, \veps ) \subseteq \M \rightarrow B(0, \veps) \subseteq T_x  \M$. For any two points $x,y$ with $d_\M(x,y) < \iota_\M$, the unique minimizing geodesic between $x$ and $y$ (i.e., a minimizer in the definition of $d_\M(x,y)$) is given by $\gamma:t\in [0,1] \mapsto \exp_{x}(tv)$ where $v = \log_x(y)$. This minimizing geodesic can be reparameterized so that it is unit speed (i.e., $\dot{\gamma}(0)$ has norm one) in which case $\gamma$ maps the interval $[0, d_\M(x,y)]$ into $\M$. From now on we will refer to this curve as the unit speed geodesic between $x$ and $y$. 


With the notion of Levi-Civita connection we can also introduce the concept of \textit{parallel transport}, an important construction that allows us to compare tangent vectors at different points on $\M$. Precisely, let $\gamma:[0,a] \rightarrow \M$ be a smooth curve on $\M$ and let $x=\gamma(0)$ and $y=\gamma(a)$. Given $v \in T_x \M$, we define $V(t) \in T_{\gamma(t)} \M$, the parallel transport of $v$ along $\gamma$, to be the (unique) solution to the equation $ \nabla_{\dot{\gamma}(t)} V = 0 $ with initial condition $V(0)=v$. We will be particularly interested in the case where $\gamma$ is the unit speed geodesic between $x$ and $y$ (for points that are sufficiently close to each other) and we will denote by $P_{xy}$ the map $P_{xy}: T_x \M \rightarrow  T_y \M $ defined as $v \in T_x \M \mapsto V(d_\M(x,y)) \in T_y \M $. We will say that $P_{xy}v$ is the parallel transport of $v$ along the geodesic connecting $x$ and $y$.
\nc

To characterize the curvature of the manifold $\M$ in a neighborhood of a point $x \in \Mc$ we can use the notion of \emph{Ricci curvature}. Formally, let $v \in T_x \Mc$ denote a unit vector and $\lbrace u_1, \dots, u_{m-1}, v \rbrace$ an orthonormal basis for $T_x\M$. We define the \textit{Ricci curvature} at $x$ along the direction $v$ as
\begin{equation}\label{eq:ric-curv}
    \Ric_x(v) := \frac{1}{m-1} \sum_{i=1}^{m-1} g({R(v,u_i)v},{u_i}) \; ,
\end{equation}
where $R(u,v)w := \nabla_u \nabla_v w - \nabla_v \nabla_u w - \nabla_{[u,v]}w$, for $[u,v]$ the Lie Bracket between $u$ and $v$. Furthermore, we can globally characterize $\Mc$'s geometry via \emph{sectional curvature}, which is given by
\begin{equation}\label{eq:sec-curv}
    K (u,v) := K_x(u,v) = \frac{g(R(u,v)u , {v} )}{|u|^2 |v|^2- (g(u,v))^2}
\end{equation}
for $u,v \in T_x \Mc$ and $x \in \Mc$; here we use the notation $|u|^2=g(u,u)$. We remark that $K$ is invariant by non-zero rescalings of the vectors $u,v$.

Let $x,y \in \M$ and let $\veps>0$ be smaller than the injectivity radius $\imath_\M$. We define $\mathcal{P}: B_\M(x,\veps) \rightarrow B_\M(y, \veps)$ the map given by
\begin{equation}
\mathcal{P}(\tilde x) =  \exp_y(P_{xy}( \log_x(\tilde x ) )).
\label{eqn:ParallelF}
\end{equation}
That is, $\tilde x$ is mapped to $x$'s tangent plane and then transported to $y$'s tangent plane along the geodesic connecting $x$ and $y$ (unique if we assume $d_\M(x,y) < \iota_\M$) to finally be mapped to $B_\M(y,\veps)$ via the exponential map at $y$. One important property of the diffeomorphism $\mathcal{P}$ that we use in the sequel, originally due to Levi-Civita, is that if we form the quadrilateral illustrated in Figure \ref{fig:LeviCivita}, then the distance between $\tilde x$ and $\tilde y$ can be precisely characterized,  up to correction terms of order 4, by the distance between $x$ and $y$ and $\M$'s sectional curvature at $x$. More precisely, we have the following result. 

\begin{proposition}[cf Proposition 6 in \citep{ollivier2009ricci}] 
    Let $\veps>0$ be a number smaller than $\iota_\M$, the injectivity radius of $\M$. Let $x, y \in \M$ be such that $d_\M(x,y) < \iota_\M$ and let $\tilde x \in B_\M(x,\veps)$ and $\tilde y := \mathcal{P}(\tilde x)$ with $\mathcal{P}$ as in \eqref{eqn:ParallelF}. Then
    \[  d_\M(\tilde x, \tilde y) = d_\M(x,y) \left( 1- \frac{(d_\M(x,\tilde x))^2}{2}(  K(v,w)  + O(\veps + d_\M(x,y) )   ) \right), \]
    where $v = \frac{\log_x(y)}{|\log_x(y)|} $ and $w= {\log_x(\tilde x)} $.
\end{proposition}

In the remainder, we will focus our discussion on \emph{embedded submanifolds} $\M$ of $\R^d$, which are defined as follows.
\begin{definition}[Embedded submanifold (see, e.g.,~\citep{boumal})]
We say that $\Mc \subseteq \R^d$ is a smooth embedded submanifold of $\R^d$ of dimension $m$ strictly less than $d$ if for every $x \in \M$ there is a ball $B(x,r) \subseteq \R^d$ and a smooth function $h_x: B(x, r) \rightarrow \R^d$ (termed \emph{defining function}), such that (i) $h_x(y)=0$ iff $y \in \Mc \cap B(x,r)$ and (ii) ${\rm rank} \; \nabla h_x(x)=d-m$.

Moreover, the inner product $g_x$ defined over $T_x \M$, the latter now seen as a subspace of $\R^d$, is the restriction of $\langle \cdot, \cdot \rangle$, the $\R^d$ inner product, to $T_x \M$.
\label{def:Embedded}
\end{definition}

\medskip

In the sequel, we use the \emph{second fundamental form} $\mathrm{I\!I}$ of the embedded manifold $\M$ in order to discuss data driven approximations to the geodesic distance $d_\M$. Let $N_x\M$ denote the normal space of $\M$, i.e., the orthogonal complement of $T_x \M$ in $\R^d$. The second fundamental form is given by the map $\mathrm{I\!I}_x: T_x \M \times T_x \M \rightarrow N_x\M$ defined by $(u,v) \mapsto (\mathrm{Id} - {\rm Proj}_{x})  (\frac{d}{dt} \restr{V(t)}{t=0})$. Here, ${\rm Proj}_x: \; \R^d \rightarrow T_x \M$ denotes the orthogonal projection onto the tangent space; $\gamma$ is a curve on $\M$ with $\gamma(0)=x$ and $\dot{\gamma}(0)=u$; $V$ is a vector field along $\gamma$ with $V(t) \in T_{\gamma(t)} \M$ satisfying $V(0)=v$.

Lastly, we define the \emph{reach} of the embedded manifold $\Mc$. Let $S \subset \mathbb{R}^d$ denote a closed subset of $\R^d$ and let $\pi_S$ be the function that maps a point $z$ (in $\R^d$) to its nearest neighbor in $S$ (if this nearest neighbor is well defined). The reach $\tau_\Mc$ of $\Mc$ is defined as the maximal neighborhood radius for which the projection $\pi_S$ is well-defined, i.e., any point that has distance at most $\tau_\Mc$ from $\Mc$ has a unique nearest neighbor on $\Mc$.

Second fundamental form and reach are notions that are quantitative was to measure the \textit{extrinsic} curvature of a manifold. In contrast, Ricci curvature and sectional curvature are \textit{intrinsic} notions. A standard way to visualize the difference between the two types of curvature is to imagine a circle drawn on a flat piece of paper, which one can think of as a one dimensional manifold $\M$ embedded in $\R^3$, and then the same circle after the paper has been rolled to form a cylinder, which can also be thought of as a one dimensional manifold $\M'$ embedded in $\R^3$. While $\M$ and $\M'$ have the same intrinsic curvature (because the paper is not stretched when rolling it), their extrinsic curvatures are different.

\begin{figure}
    \centering
\includegraphics[trim={5cm 25cm 10cm 12cm},width = 1.1\linewidth]{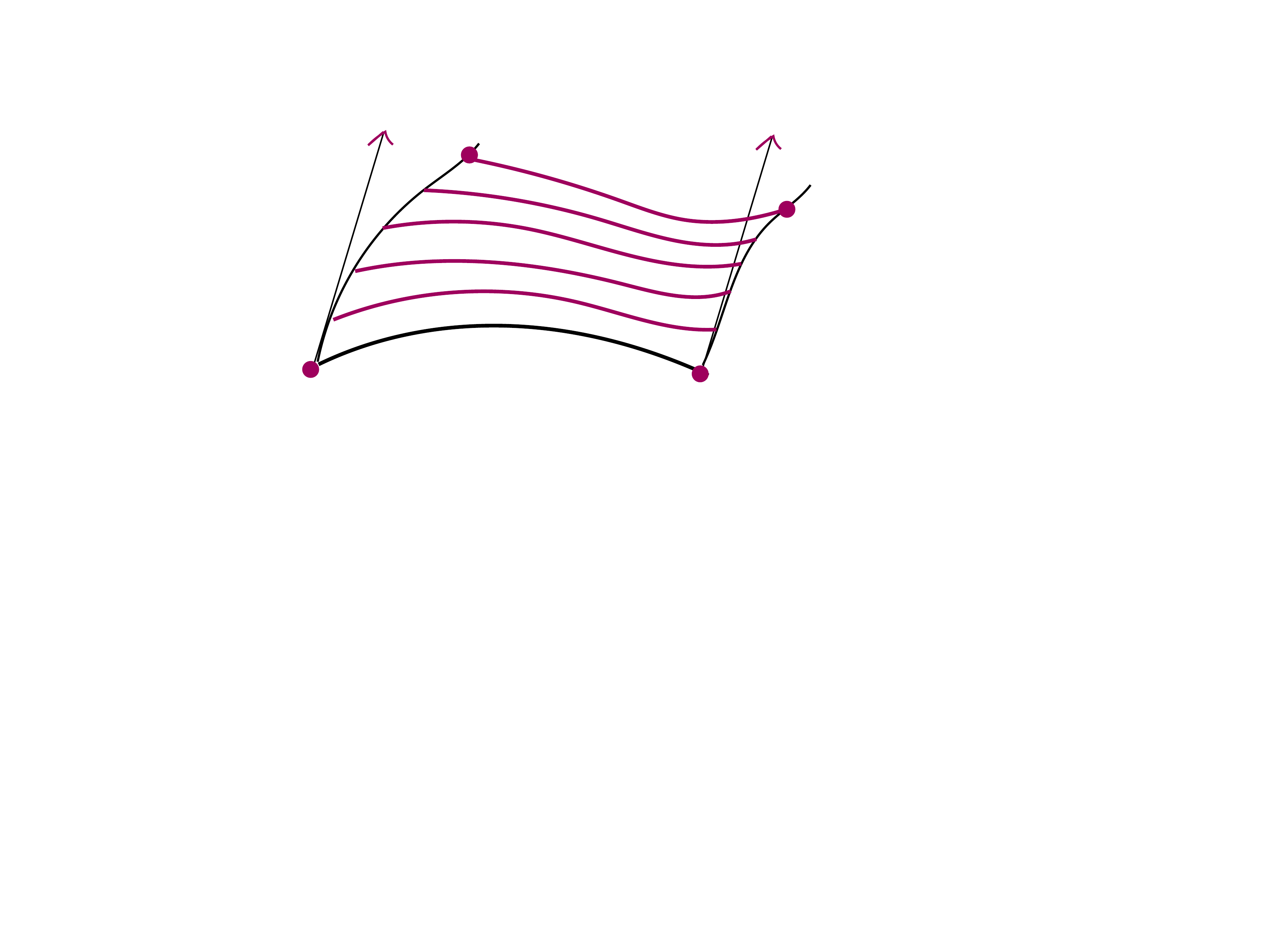}
\put(-255, 125 ){$\tilde x$}
\put(-112, 90 ){$\tilde y$}
\put(-315, 20 ){$x$}
\put(-157, 14 ){$y$}
    \caption{Levi-Civita parallelogram. All curves represent geodesics and the straight lines represent tangent vectors.}
    \label{fig:LeviCivita}
\end{figure}

\nc

\subsection{Ollivier's Ricci curvature }
\label{sec:Ollivier}

Let $(\U,d)$ be a metric space and  let $\mu_1, \mu_2$ be two (Borel) probability distributions over $\U$. Recall that the \emph{1-Wasserstein distance} between $\mu_1, \mu_2$ is defined as
\begin{equation}\label{eq:W1-dist}
    W_1(\mu_1, \mu_2) := \inf_{\pi \in \Gamma(\mu_1,\mu_2)} \int_{(x,y) \in \U \times \U} d(x,y) d\pi(x,y) \; ,
\end{equation}
where $\Gamma(\mu_1,\mu_2)$ is the set of measures on $\U \times \U$ with first and second marginals equal to $\mu_1$ and $\mu_2$, respectively. Let $\mathbf{m}$ be a Markov kernel over $\Xc$, i.e., $\mathbf{m}$ is a (measurable) collection $\{m_x \}_{x \in \U}$ of probability measures over $\U$. The Ollivier Ricci curvature associated to the triplet $(\U,d, \mathbf{m})$ in direction $(x,y)$ is defined as (see~\citep{ollivier2009ricci}):
\begin{equation}\label{eq:ORC-metric}
    \kappa(x,y) := 1 - \frac{W_1(m_x,m_y)}{d(x,y)}.
\end{equation}
Notice that, in general, the notion of Ollivier Ricci curvature only requires the structure of a metric space endowed with a (discrete-time) random walk. In the remainder of this section we will consider the case where $\U$ is assumed to be a Riemannian manifold and discuss the relation between \eqref{eq:ORC-metric} and the geometric notion of Ricci curvature introduced in section \ref{sec:DiffGeometry}. To do this we first need some definitions.

\begin{definition}[Ollivier's Ricci curvature on manifolds~\citep{ollivier2009ricci}]
Let $d_{\Mc}(x,y)$ denote the geodesic distance in $\M$ between $x$ and $y$ and let $B_\Mc(x,\veps)$, $B_\Mc(y,\veps)$ be the closed balls of radius $\veps$ (a fixed parameter) around $x$ and $y$, respectively (termed \emph{Ollivier balls}). Let further
\begin{align}
    \mu_x^\Mc(z) &= \frac{\vol(dz) \lfloor_{B_{\M}(x,\veps)}}{\vol(B_{\Mc}(x,\veps))} \\
    \mu_y^\Mc(z) &= \frac{\vol(dz)\lfloor_{B_{\M}(x,\veps)}}{\vol(B_{\Mc}(y,\veps))}
\end{align}
denote uniform measures on those neighborhoods. We define \emph{Ollivier's Ricci curvature} between $x$ and $y$ as
\begin{equation}
    \k_\Mc (x,y) = 1 - \frac{W_1(\mu_x^\Mc, \mu_y^\Mc)}{d_\Mc(x,y)} \; .
\end{equation}
\label{def:OllivierBalls}
\end{definition}
\noindent Ollivier showed the following fundamental relationship between $\Ric$ and $\k_\Mc$:

\begin{theorem}[Ollivier~\citep{ollivier2009ricci}]\label{thm:ollivier1}
\begin{equation}
    \Big\vert \k_\Mc (x,y) -  \frac{\veps^2}{2(m+2)} \Ric_x(v)  \Big\vert \leq  \left(  C \veps^2 d_\Mc(x,y) + C' \veps^3
    \right)\; ,
\end{equation}
where $y$ is a point on the geodesic from $x$ in direction $v$ (of norm one) and $\veps$ is the radius of the Ollivier balls. $C, C'$ are constants independent of $m$.
\end{theorem}

This theorem makes precise the intuition that random walks in $\M$ starting at nearby points draw closer together if the Ricci curvature is positive and further apart if the Ricci curvature is negative, provided the walks are suitably coupled. It also provides the motivation for extending the definition of Ricci curvature to arbitrary triplets $(\U , d , \mathbf{m})$.


In section \ref{sec:ProofTheorem1}, we present the main steps in the proof of Theorem \ref{thm:ollivier1}. This will give us the opportunity to introduce some key estimates and constructions that we use later in section \ref{sec:NonAsymptoticGuarantees} when proving our main results.

%

\subsection{Curvature on Random Geometric Graphs}
\label{sec:DiscreteOllivier}

We recall the notion of a \emph{random geometric graph} (short: RGG) on $\M$. Let $\vol$ denote $\Mc$'s volume form and let $\mu$ be the uniform probability measure over $\M$ defined by
\[\mu (A) := \frac{\int_{A} d\vol(x)  }{\int_{\M} d\vol(x)} \]
for all Borel measurable subsets $A$ of $\M$. Let $\X = \lbrace x_i \rbrace_{i=1}^n$ be a collection of i.i.d. samples from $\mu$. In the sequel, we use $\mu_n$ to denote the empirical measure associated to these data points. Namely,
\[\mu_n:= \frac{1}{n}\sum_{i=1}^n \delta_{x_i}. \]

We construct a random geometric graph $G_\veps = (\X, w_\veps)$ by connecting any pair of samples $x,y \in \X$ with an edge whenever their geodesic distance, or an approximation thereof, is less than $\veps$. More precisely, we'll set $w_\veps$ to be either
 \begin{equation}
     w_{\veps}(x, y) = w_{\veps, \M} (x,y):= \begin{cases} 1 & \text{ if } d_\M(x,y) \leq \veps \\ 0 & \text{else}, \end{cases} 
     \label{eqn:WeightsManifold}
 \end{equation}
 or 
 \begin{equation}
   w_{\veps}(x,y) = w_{\veps, \X}(x,y):= \begin{cases} 1 & \text{ if } \hat{d}_g(x,y) \leq \veps \\ 0 & \text{else}. \end{cases}  
   \label{eqn:WeightsDataDriven}
 \end{equation}

The first setting, which uses the geodesic distance on $\M$, is only reasonable in applied settings where the manifold $\M$ is known. In contrast, the second setting only presupposes having access to a function $\hat{d}_g$ that serves as a data-driven local approximation for the geodesic distance $d_\M$. In section \ref{sec:hatd_g} we discuss some conditions that we need to impose on this approximation and highlight the need to work with approximations for $d_\M$ of high enough order if one desires to recover precise curvature information of the underlying manifold $\M$ from the graph $G_\veps$ as $n \rightarrow \infty$.


 Analogous to the continuous case, we can define Ollivier's Ricci curvature on the RGGs introduced above. In order to do so, we recall that we need two ingredients: a random walk over $\X$ and a notion of distance over $\X$.  
 
First, we introduce the graph Ollivier balls
 \begin{equation}
     B_G(x,\veps) := \lbrace z \in \X: \; w_\veps(x,z) =1 \rbrace \quad \forall x \in \X \; ,
 \end{equation}
  and we consider the family $\{ \mu_x^G \}_{x\in \X}$ of uniform distributions
 \begin{equation}
     \mu_x^G (z) = \frac{1}{\vert B_G(x,\veps) \vert} \; , \quad z \in B_G(x,\veps) \;. 
 \end{equation}
 Observe that the measures $\mu_x^G$ define the transition probabilities for a random walk on the RGG. The generator of the discrete-time Markov chain with transition probabilities given by $\{ \mu_x^G \}_{x\in \X}$ is known in the literature as the \textit{random walk graph Laplacian}; see \citep{VonLuxburg2007}.

\begin{remark}
\label{rem:GeneralizedOllivierBalls}
Notice that the construction of the ball $B_G(x,\veps)$ and its associated probability measure over $\X$, $\mu_x^G$, continues to make sense for any arbitrary base point $x \in \M$, regardless of whether $x$ is a data point in $\X$ or not. This observation will be used in our theoretical analysis in subsequent sections.
\end{remark}

The second ingredient needed to define Ollivier's Ricci curvature over $G_\veps$ is a distance function $d_G$ over $\X$. Two specific choices for $d_G$, one useful when $d_\M$ is unknown and the other useful when $d_\M$ is known, will be discussed in section \ref{sec:GeodesicDistancesX}; both choices will endow $\X$ with a suitable geodesic metric space structure. 
Either way, once $d_G$ has been fixed, we can define, relative to the family of measures $\{\mu_x^G\}_{x \in \X}$ (which in turn we recall depends on the choice of $w_\veps$), the Ollivier Ricci curvature between points $x, y \in \X$ as
 \begin{equation}
     \k_G (x,y):= 1 - \frac{W_{1,G} (\mu_x^G, \mu_y^G)}{d_G(x,y)}.
     \label{eqn:GraphCurvature}
 \end{equation}
 In the above, $W_{1,G}(\mu_x^G, \mu_y^G)$ is the $1$-Wasserstein distance induced by the metric $d_G$ over $\X$. Precisely,
\begin{equation}\label{eq:W1G-dist}
    W_{1,G} (\mu_x^G, \mu_y^G) = \min_{\pi \in \Gamma(\mu_x^G,\mu_y^G)} \int d_G(\tilde x, \tilde y) d\pi(\tilde x , \tilde y).
\end{equation}
 

\nc

 \section{Distance functions over $\X$}
 \label{sec:APproximationGeodesics}

 As discussed at the end of the previous section, to associate a notion of Ollivier's Ricci curvature to the data set $\X$ we need to introduce two functions $\hat{d}_g$ and $d_G$ that are used to specify a Markov chain and a metric over $\X$. In this section we provide more details on these two functions.

  \subsection{Assumptions on $\hat{d}_g$}
\label{sec:hatd_g}
 
 In one of the settings that we study in this paper we assume that we have no access to pairwise geodesic distances between points $x,y$ in our data set. In this setting, in order to recover the manifold's intrinsic curvature information from the data we will assume to have access to an oracle, data-driven estimator of $d_\M$, denoted $\hat{d}_g$, satisfying the following conditions. 

\begin{assumptions}
    \label{assump:hatdg} 
    
    The function $\hat{d}_g : \M \times \M \rightarrow [0, \infty) $ is assumed to be a symmetric function satisfying, with probability at least $1 - C\exp( - \zeta(\beta, n,  \veps))$, the following conditions:  
\begin{enumerate}
    \item For every pair $x,y \in \M$ satisfying $ c \veps \leq d_\M(x,y)  \leq  C \veps   $ or  $ c \veps \leq \hat{d}_g(x,y)  \leq  C \veps   $ we have
    \begin{equation}
    |d_\M(x,y) - \hat{d}_g(x,y)| \leq   C_1 \beta \veps^3 + C_2\veps^4.
\label{eqn:ApproxDistance_HigherOrder}
\end{equation}

\item For every pair $x,y \in \M$ the following implication holds:
\begin{equation}
\hat{d}_g(x,y) \leq c \veps \quad  \Longrightarrow
 \quad d_\M(x,y) \leq \frac{4}{3}  c\veps.  
 \label{eq:SmallScaledHat}
\end{equation}

\end{enumerate}
Here, $\zeta(n, \beta,\veps)$ is assumed to be of the form $C n^{q_1} \veps^{q_2} \beta^{q_3} $ for positive powers $q_1, q_2, q_3>0$. In particular, $\zeta(n, \beta,\veps) \rightarrow \infty$ as $n \rightarrow \infty$, whenever $\beta >0$ and $\veps>0$ are fixed.

\end{assumptions}

\begin{remark}
Observe that in Assumption \ref{assump:hatdg} we only require the function $\hat{d}_g(x,y)$ to satisfy symmetry but none of the other axioms defining a distance function. 
\end{remark}

The first condition in Assumption \ref{assump:hatdg} states that, with high probability, $\hat{d}_g$ must approximate $d_\M$ locally with an error of order four, at least as long as the distance between points is not too small. In general, one should not expect the same type of error estimate for $d_\M$ at very small length scales, which is why we instead require condition 2 for nearby points, a much milder assumption. 

It is worth highlighting that the Euclidean distance is not a valid choice for $\hat{d}_g$, since its error of approximation of $d_\M$ is only of order three; see the discussion in Appendix \ref{sec:EuclideanVsGeodesic} for more details. On the other hand, a straightforward computation, which we show in detail in Appendix \ref{sec:EuclideanVsGeodesic}, reveals that
\begin{equation}
\label{eq:HigherOrderApprox}
d_\M(x,y) = |x-y| + \frac{1}{24} \langle \ddot{\gamma}(0), \ddot{\gamma}(0) \rangle  |x-y|^3 + O(|x-y|^4), \quad x, y \in \M,  
\end{equation}
where $\ddot{\gamma}(0)$ is the acceleration (in the ambient space $\R^d$) at time $0$ of the unit speed geodesic connecting $x$ and $y$. The above formula thus suggests estimating the term $\langle \ddot \gamma(0), \ddot \gamma(0) \rangle$ to define a data-driven approximation $\hat{d}_g$ for $d_\M$. In turn, to estimate $\langle \ddot \gamma(0), \ddot \gamma(0) \rangle$ it suffices to estimate $\M$'s second fundamental form, as one can write 
\[ \ddot{\gamma}(0) = \mathrm{I\!I}_x(\dot \gamma(0) , \dot \gamma(0)).\]
The above discussion motivates introducing the function 
\begin{equation}
  \hat{d}_g(x,y) = |x-y| + \frac{1}{48}\left( |\hat{\mathrm{I\!I}}_{xy}|^2 + |\hat{\mathrm{I\!I}}_{yx}|^2 \right) |x-y|^3, \quad x,y \in \X,  
  \label{eqn:EstimatedPreDistance}
\end{equation}
where $\hat{\mathrm{I\!I}}_{xy}$ is an estimator, built from data, for 
\[\mathrm{I\!I}_{xy}:= \mathrm{I\!I}_x(\dot{\gamma}(0), \dot{\gamma}(0))= \mathrm{I\!I}_x \left( \frac{\log_x( y)}{|\log_x(y)|},  \frac{\log_x( y)}{|\log_x(y)|}  \right);\]
observe that $\hat{d}_g$ defined in this way is a symmetric function, as required in Assumption \ref{assump:hatdg}. In Appendix~\ref{sec:EstimatesSecondFundForm} we discuss some existing approaches in the literature for building the estimators $\hat{\mathrm{I\!I}}_{xy}$. 
The relevant observation is that if $|\mathrm{I\!I}_{xy}|^2$ can be approximated using $|\hat{\mathrm{I\!I}}_{xy}|^2$ within error $\beta$, then we could use the fact that $ |\mathrm{I\!I}_{xy}|^2 = |\mathrm{I\!I}_{yx}|^2 + O(|x-y|) $, given that the manifold $\M$ is smooth, to conclude that
\[|\hat{d}_g(x,y) - d_\M(x,y)| \leq C_1 \beta |x-y|^3 + C_2 |x-y|^4, \]
obtaining in this way a higher order approximation for $d_\M$ than the Euclidean distance.

There are several other ways in which one can construct a data-driven approximation for $d_\M$ (see more discussion in section \ref{sec:Literature}) and in the sequel we will not commit to using any particular form for it.  
On the other hand, at this stage we wanted to give some particular attention to the choice \eqref{eqn:EstimatedPreDistance} to highlight that, when combining with our main theoretical results in section \ref{sec:MainResults} (specifically Theorems \ref{thm:consistency-non-asymp-geo} and \ref{thm:GlobalBounds2}), by using \eqref{eqn:EstimatedPreDistance} to define Ollivier's Ricci curvature over the data set $\X$ we would be turning estimators of \textit{extrinsic} curvature into estimators for \textit{intrinsic} curvature. This conceptual implication is not a priori obvious.

\subsection{Definitions of $d_G$}
\label{sec:GeodesicDistancesX}
 To define a geodesic distance $d_G$ over $G=(\X,w_\veps) $ (interpreting $w_\veps$ as either \eqref{eqn:WeightsManifold} or as \eqref{eqn:WeightsDataDriven})  we first introduce the following ``pre-distance" functions:

 \[ \tilde d_{G,\M}(x,y) := \begin{cases} 0, & \text{if } x=y, \\
 \delta_0 \psi \left( \frac{d_\M(x,y)}{\delta_0} \right), & \text{if } 0<d_\M(x,y)\leq \delta_1,  \\
 +\infty, & \text{otherwise},
 \end{cases} \]
and
 \[ \tilde d_{G,\X }(x,y) := \begin{cases} 0, & \text{if } x=y, \\ \delta_0 \psi \left( \frac{\hat d_g(x,y)}{\delta_0} \right), & \text{if } 0<\hat{d}_g(x,y)\leq {\delta}_1,\\
 +\infty, & \text{otherwise},
 \end{cases} \]
 \nc 
 where $\hat{d}_g$ satisfies Assumption \ref{assump:hatdg}. In the above, we use parameters $\delta_0<\delta_1$ that in terms of the parameter $\veps$ will be written as 
 \begin{equation}
    \delta_0 := c_0 \veps, \quad \delta_1 := c_1 \veps
    \label{def:Delta_0Delta_1}
 \end{equation}
 where $c_0$ is a fixed but small enough number and $c_1 $ is fixed but large enough; see more details below. Finally, the profile function $\psi$ appearing in both definitions is assumed to satisfy the following conditions: 
 \begin{assumptions}
 \label{assumpPsi}
   The function $\psi:[0, \infty) \rightarrow [0, \infty)$ satisfies the following:
   \begin{enumerate}
\item $\psi$ is $C^2$, non-decreasing, and convex.
\item $\psi(t)= t$ for all $t \geq 1$.
\item $\psi(0) >0 $ and $\psi'(0)>0$.
   \end{enumerate}
 \end{assumptions}
In Figure \ref{fig:ProfilePsi} we provide an example of an admissible profile function $\psi$. As we discuss in detail in Remark \ref{rem:OnPsi}, Assumption \ref{assumpPsi} on the profile function guarantees that the geometry of the RGG is suitably glued together when moving between two length-scales at which the RGG exhibits two different qualitative (and quantitative) behaviors. Indeed, at very small length-scales the restriction of the RGG $G_\veps$ around any given point behaves like a complete graph, while at larger scales the RGG exhibits manifold-like behavior. In order to make statements about global curvature lower bounds for the RGG (see our main results in section \ref{sec:MainGlobal}), that is, to say that \textit{for all} $x, y \in \X$ the value of $\kappa_G$ is lower bounded  by a constant, we use the bending of $\psi$ as it approaches the origin as a helpful technical tool that allows us to control the error estimates in the regions of transition from complete graph to manifold-like behavior. This bending comes at the price of introducing a bias in the estimated curvature lower bound (see for example \eqref{eqn:LowerboundTheorem1} in Theorem \ref{eqn:LowerboundTheorem1}), which in some cases may not be too important (for example, when stating contractivity properties of heat kernels, as we do in section \ref{sec:LaplaciansandRegression}). We anticipate that removing this bias requires a much more delicate analysis as currently we use two very different approaches for estimating lower bounds in each of the regimes described above.

\begin{figure}
    \centering
\includegraphics[trim={0cm 0cm 0cm 0cm},width = 0.5\linewidth]{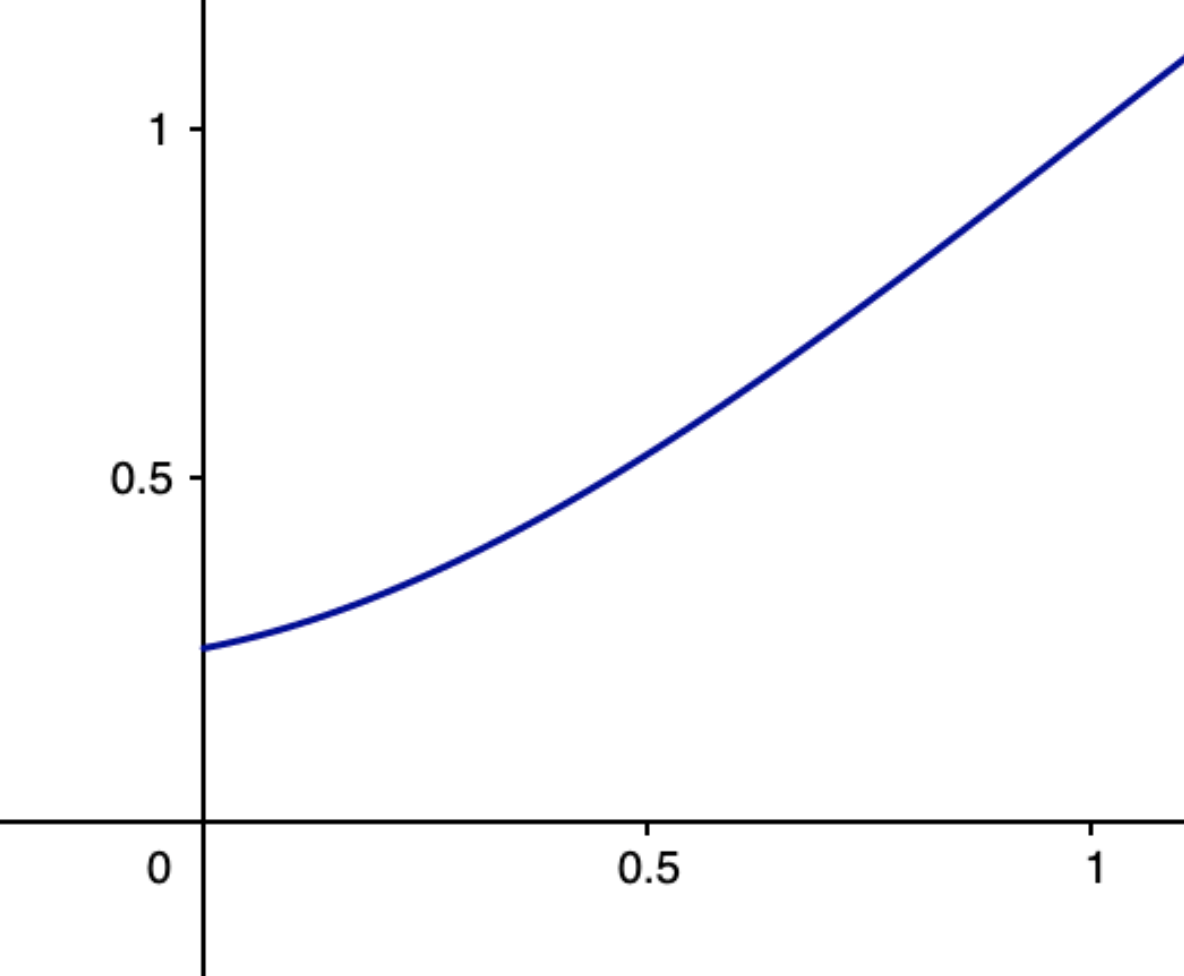}
\caption{Plot of the function $\psi(t):= \begin{cases} \frac{1}{4} (1-t)^3 + t & \text{ if } 0 \leq  t \leq 1 \\ t & \text{ if } t >1 \end{cases},$ which satisfies all the required conditions in Assumption \ref{assumpPsi}. \nc  }
    \label{fig:ProfilePsi}
\end{figure}

 \medskip
 
We may now use the pre-distance functions to define two geodesic distances over $\X$: 
\begin{equation}
 d_{G,\M}(x, y) := \inf_{\substack{n \in \N,\: \{ x_i \}_{i=0}^n \subseteq \X \\  x_0=x,\:  x_n=y  }} \sum_{i=0}^{n-1} \tilde d_{G,\M}(x_i, x_{i+1}),
 \label{eq:DistancefromGeodesic}
 \end{equation}
and 
\begin{equation}
\label{eq:DistancefromDataDriven}
d_{G,\X}(x, y) := \inf_{\substack{n \in \N,\: \{ x_i \}_{i=0}^n \subseteq \X \\  x_0=x,\:  x_n=y  }} \sum_{i=0}^{n-1} \tilde d_{G,\X}(x_i, x_{i+1}).
\end{equation}
Note that, in contrast to $d_{G,\M}$, the function $d_{G, \X} $ is completely data-driven and thus in principle more useful in applications where $d_\M$ is unknown. As we show below, both $d_{G,\M}$ and $d_{G, \X}$ are indeed distance functions over $\X$ and both of them induce Olivier Ricci curvature functions that recover $\M$'s Ricci curvature in the large data limit.

 \begin{lemma}
 \label{lem:d_GDistance}
Let $d_G = d_{G,\M}$ and $\tilde{d}_G = \tilde{d}_{G,\M}$, or $d_G = d_{G,\X}$ and $\tilde{d}_G = \tilde{d}_{G,\X}$. In either case the function $d_G$ is a distance over $\X$. Moreover, for every $x,y \in \X$ there exists a finite sequence $x_1, \dots, x_k \in \X$ such that:
 \begin{enumerate}
     \item $x_1=x$, and $x_k=y$.
     \item $d_G(x_i, x_{i+1}) = \tilde{d}_{G}(x_i , x_{i+1}) \leq \delta_1$.
     \item $d_G(x,y)= \sum_{i=1}^{k-1} d_{G}(x_i, x_{i+1}) $.
 \end{enumerate}
 \end{lemma}
 \begin{proof}
 The fact that $d_G= d_{G,\M}$ is a distance function follows directly from the definitions. Likewise, we can see that $d_G= d_{G, \X}$ is a distance function thanks to the fact that $\hat{d}_g$ is symmetric, according to Assumption \ref{assump:hatdg}.

To prove the second part, notice that for an arbitrary pair $x,y \in \X$ we can find a path  $x_1, \dots, x_k \in \X$ with $x_1=x$ and $x_k=y$ such that $ d_G(x,y)= \sum_{i=1}^{k-1}  \tilde d _G (x_i , x_{i+1})$. Now, by definition of $d_G(x_i, x_{i+1})$, we have $\tilde d_G (x_i, x_{i+1}) \geq d_G(x_{i}, x_{i+1})$ for every $i=1, \dots, k-1$. If the inequality was strict for at least one $i$, then we would actually be able to build a path connecting $x,y$ whose length is strictly smaller than $d_G(x,y)$, which would be a contradiction. It follows that $d_G(x_i, x_{i+1})= \tilde d_G(x_i, x_{i+1})$ for all $i=1, \dots, n$.

 
 \end{proof}

\begin{lemma}
\label{lem:CurvatureGeodesicProp}
Let $\kappa  \in \R$ and suppose that 
\begin{equation}
   1 - \frac{W_{1,G}(\mu_x^G, \mu_y^G)}{d_G(x,y)}  \geq \kappa,
   \label{eq:LowerBoundCurvAux}
\end{equation}
for every $x,y \in \X$ satisfying $d_G(x,y) = \tilde d_{G}(x,y) \leq \delta_1$. Then \eqref{eq:LowerBoundCurvAux} holds for any pair $x,y \in \X$ 
\end{lemma}

\begin{proof}
  Thanks to Lemma \ref{lem:d_GDistance}, the proof is just as in [Prop.~19,~\citep{ollivier2009ricci}]. The emphasis here, however, is the fact that we only need to check the inequality for pairs $x,y$ for which the distance function $d_G$ and the pre-distance function $\tilde d _G $ coincide.   
\end{proof}

We finish this section with the following inequalities relating the metrics $d_{G,\X}$, $d_{G, \M}$, and $d_\M$. These inequalities will only be used later on in section \ref{sec:LocalConsistency} when studying curvature upper bounds.

\begin{proposition}
\label{prop:ComparisonMetrics} 
Under Assumption \ref{assump:hatdg}, for all small enough $\veps$ and $\beta$ the following holds:
\begin{equation}
    d_{G,\M}(x,y) \geq  d_\M(x,y), 
    \label{eq:Comparisson1}
\end{equation}
for all $x, y \in \X$, 
and, with probability at least $1- C\exp(- \zeta(n, \beta,\veps))$ we have
\begin{equation}
  (1+C(\beta \veps^2+\veps^3)) d_{G,\X}(x,y) \geq   d_\M(x,y).  
   \label{eq:Comparisson2}
\end{equation}
for all $x, y \in \X$. Moreover, if $x,y \in \X$ are such that $2\delta_0 \leq d_\M(x,y) \leq \frac{1}{2}\delta_1$, where $\delta_0$ and $\delta_1$ are as in the definition of $d_{G,\X}$ and $d_{G, \M}$, then
\begin{equation}
     d_{G,\X}(x,y) \leq  d_\M(x,y)( 1+ C( \beta \veps^2 + \veps^3))
      \label{eq:Comparisson3}
\end{equation}
and
\begin{equation}
     d_{G,\M}(x,y) \leq  d_\M(x,y).
      \label{eq:Comparisson4}
\end{equation}

\end{proposition}
\begin{proof}
Inequality \eqref{eq:Comparisson1} is immediate from the triangle inequality for $d_\M$ and the fact that $\tilde{d}_{G,\M}(x_{i},x_{i+1}) \geq   \delta_0\psi( d_\M(x_i,x_{i+1}) /\delta_0 ) \geq d_\M(x_i, x_{i+1}) $ for any two data points $x_i, x_{i+1}$ with $d_\M(x_i, x_{i+1}) \leq \delta_1$, since necessarily $\psi(t) \geq t$ for all $t \geq 0$.

To prove inequality \eqref{eq:Comparisson2}, consider an arbitrary discrete path $\{ x_i \}_{i=0}^n$ connecting $x$ and $y$ such that $\tilde{d}_{G, \X}(x_i, x_{i+1}) <\infty$ for all $i$. Then
\begin{align*}
\sum_{i=0}^{n-1} \tilde{d}_{G, \X}(x_{i}, x_{i+1}) & = \sum_{i=0}^{n-1} \delta_0\psi(\hat{d}_{g}(x_{i}, x_{i+1})/\delta_0) 
\\& =   \sum_{i \in A}\delta_0\psi(\hat{d}_{g}(x_{i}, x_{i+1})/\delta_0) +  \sum_{i \in B}\delta_0\psi(\hat{d}_{g}(x_{i}, x_{i+1})/\delta_0)
\\& \geq    \sum_{i \in A}\hat{d}_{g}(x_{i}, x_{i+1}) +  \sum_{i \in B}\delta_0\psi(\hat{d}_{g}(x_{i}, x_{i+1})/\delta_0)
,
\end{align*}
where $A:= \{ i \text{ s.t. } \hat{d}_g(x_i, x_{i+1}) \leq  \frac{3}{4} \psi(0)\delta_0  \} $ and $B:= \{ i \text{ s.t. } \hat{d}_g(x_i, x_{i+1}) >  \frac{3}{4} \psi(0) \delta_0 \} $. For an $i \in A$, we can use 2 in Assumption \ref{assump:hatdg} to conclude that $d_\M(x,y) \leq \delta_0 \psi(0) \leq  \tilde{d}_{G, \X}(x_i, x_{i+1})$, whereas for an $i$ in $B$ we can use 1 in Assumption \ref{assump:hatdg} to bound the difference between $\hat{d}_g(x_i, x_{i+1})$ and $d_\M(x_i, x_{i+1})$. In particular,
\begin{align*}
 \sum_{i=0}^{n-1} \tilde{d}_{G, \X}(x_{i}, x_{i+1}) & \geq  \sum_{i \in A}  ( \hat{d}_{g}(x_{i}, x_{i+1})  - d_\M(x_i, x_{i+1}) +  d_\M(x_i, x_{i+1}) )
 \\& +   \sum_{i \in B}   d_\M(x_i, x_{i+1})
\\& \geq \sum_{i=0}^{n-1}  (1 - C_1 \beta \veps^2 - C_2 \veps^3 )  d_\M(x_i, x_{i+1})
\\& \geq (1 - C_1 \beta \veps^2 - C_2 \veps^3 )  d_\M(x, y) ,
\end{align*}
where the last inequality follows from the triangle inequality for $d_\M$. From this we deduce that
\[ (1 + C (\beta \veps^2+ \veps^3)) d_{G, \X}(x,y) \geq d_\M(x,y), \]
for some constant $C$.

To prove \eqref{eq:Comparisson3}, notice that, under the assumption that $\veps$ is small enough, we can guarantee, thanks to \eqref{eqn:ApproxDistance_HigherOrder}, that $ \hat{d}_g(x,y) \in [\delta_0, \delta_1)  $. This means that 
\[  d_{G, \X}(x,y) \leq \hat{d}_g(x,y).   \]
In turn, applying \eqref{eqn:ApproxDistance_HigherOrder} again we can upper bound $\hat{d}_g(x,y)$ from above by $(1+ C(\beta \veps^2+\veps^3) ) d_\M(x,y)$. Inequality \eqref{eq:Comparisson3} now follows.

Inequality \eqref{eq:Comparisson4} is obvious from the definition of $d_{G, \M}$ and the assumption on $d_\M(x,y)$. 

\end{proof}

\nc

 \section{Main Results and Discussion}

 \label{sec:MainResults}

We are now ready to state our main results. Throughout this section we will make the following assumptions on the scale of the parameters that determine $G$, the estimator $\hat{d}_g$, and the discrete curvature $\kappa_G$.

\begin{assumptions}
\label{assump}
We assume that the following relations hold:
\begin{enumerate}
\item  $c_1 \geq 2 + 4c_0$, where $c_0$ and $c_1$ are as in \eqref{def:Delta_0Delta_1}.
\item $\veps$ and $\beta$ are sufficiently small and $n $ is sufficiently large.
\item The ratio $\frac{\log(n)^{p_m}}{n^{1/m} \veps^3}$ is sufficiently small, where $p_m$ is a dimension dependent quantity: $p_m=3/4$ when $m=2$, and $p_m= 1/m$ when $m\geq 3$. In the case $m=1$, we assume the ratio $\frac{\log(n)^{1/2}}{n^{1/2} \veps^3}$ to be sufficiently small.
\end{enumerate}
\end{assumptions}

The third item in Assumption \ref{assump} can be succinctly described as requiring that the $\infty$-OT transport distance between $\mu$ and $\mu_n$ (see \eqref{eq:W-inf} for a definition) is much smaller than $\veps^3$, according to Theorem \ref{thm:inftyOTEmpiricalBound}, where the $\infty$-OT distance between $\mu$ and $\mu_n$ gets bounded with very high probability.

 \subsection{Poinwise Consistency}
\label{sec:MainPointwiseState}
 
Van der Hoorn et al.~\citep{van2021ollivier} were the first to analyze the \emph{pointwise consistency} of some form of discrete Ricci curvature on an RGG. They give \emph{asymptotic} convergence guarantees of the following form:

 \begin{proposition}[Pointwise consistency~\citep{van2021ollivier}]\label{thm:consistency-asymp}
 Suppose that $d_G$ is taken to be $d_\M$ in \eqref{eqn:GraphCurvature} and that $\veps= \veps_n \sim n^{-\alpha}$ and $\delta=\delta_n= n^{-\beta}$ with $0< \beta \leq \alpha$ and $\alpha + 2\beta < \frac{1}{m}$. Then we have, in expectation over the RGG ensemble $G=(\X, w_{\veps,\M})$
 \begin{equation}
     \lim_{n \rightarrow \infty}\E \Bigg[ \left|\frac{1}{\delta^2} \kappa_G - \frac{1}{2(m+2)} \Ric_x(v) \right| \Bigg] =0 \; ,
 \end{equation}
 where $y= y_n$ is the point on the geodesic in direction $v$ starting at $x$ with $\delta=d_\Mc(x,y)$. 
 \end{proposition}

\begin{remark}
It is worth pointing out that in the definition of $\kappa_G$ in \citep{van2021ollivier} we have $d_G = d_{\M}$, whereas here we work with $d_{G, \M}$. We will be able to obtain global lower bounds for our induced curvature using $d_{G,\M}$, while we cannot derive those lower bounds for $d_\M$.
 \end{remark}
 %
 The authors of \citep{van2021ollivier} provide a second analysis which only requires access to pairwise Euclidean distances in the ambient space, but this result relies on a crucial auxiliary result, which is currently only available in dimension 2. 

 In what follows we address Problem \ref{prob1} and  provide \emph{non-asymptotic} error bounds for the approximation of $\M's$ Ricci curvature from our notions of discrete Ricci curvature. Our first result is stated in the setting where we have access to $d_\M$.

%

 \begin{theorem}[Pointwise consistency (access to geodesic distances)] \label{thm:consistency-non-asymp}
Let $\M$ be an $m$-dimensional, compact, boundaryless, connected, smooth manifold embedded in $\R^d$. Let $\X= \{ x_1, \dots, x_n \}$ consist of i.i.d. samples from the uniform distribution on $\M$. Let $w_\veps= w_{\veps, \M}$ be defined according to \eqref{eqn:WeightsManifold}, let $d_G= d_{G,\M} $ be the metric defined in \eqref{eq:DistancefromGeodesic} for a profile function $\psi$ satisfying Assumptions \ref{assumpPsi}, and let $\kappa_G$ be the Ollivier Ricci curvature  induced by these choices of Ollivier balls and metric (see \eqref{eqn:GraphCurvature}).

Under Assumption \ref{assump}, for every $s >1$ there is a constant $C$ such that, with probability at least $1- C n^{-s}$, we have
 \begin{equation}
     \left| \frac{\kappa_G(x,y)}{\veps^2} -   \frac{ \Ric_x(v)}{2(m+2)} \right| \leq  C \left( \veps  +   \frac{\log(n)^{p_m}}{ n^{1/m}\veps^3} \right),
 \end{equation}
 for all $x,y \in \X$ satisfying $2\delta_0 \leq d_\M(x,y) \leq \frac{1}{2} \delta_1$, where $\delta_0$ and $\delta_1$ are defined in \eqref{def:Delta_0Delta_1}. In the above, we use $v$ to denote the vector $\frac{\log_x(y)}{|\log_x(y)|} \in T_x \M$, and $p_m =3/4$ if $m=2$, while $p_m=1/m$ if $m\geq 3$. In case $m=1$, the above statement continues to hold after substituting $\frac{(\log(n))^{p_m}}{n^{1/m}}$ with $\frac{(\log(n))^{1/2}}{n^{1/2}}$.
 \end{theorem}

A second result, which only assumes access to a sufficiently sharp data-driven approximation $\hat{d}_g$ of the geodesic distance $d_\M$ (see Assumption \ref{assump:hatdg}), is stated below. 

 \begin{theorem}[Pointwise consistency (approximate geodesic distances)]\label{thm:consistency-non-asymp-geo}
Let $\M$ be an $m$-dimensional, compact, boundaryless, connected, smooth manifold embedded in $\R^d$. Let $\X= \{ x_1, \dots, x_n \}$ consist of i.i.d. samples from the uniform distribution on $\M$. Let $w_\veps= w_{\veps, \X}$ be defined according to \eqref{eqn:WeightsDataDriven}, for a data-driven approximation $\hat{d}_g$ of $d_\M$ satisfying Assumption \ref{assump:hatdg}. Let $d_G= d_{G,\X} $ be the metric defined in \eqref{eq:DistancefromDataDriven} for a profile function $\psi$ satisfying Assumptions \ref{assumpPsi}, and let $\kappa_G$ be the Ollivier Ricci curvature  induced by these choices of Ollivier balls and metric (see \eqref{eqn:GraphCurvature}).

Under Assumption \ref{assump}, for every $s >1$ there is a constant $C$ such that, with probability at least $1- C n^{-s} - C\exp(- \zeta(n, \beta,\veps))$, we have
 \begin{equation}
      \left| \frac{\kappa_G(x,y)}{\veps^2} -   \frac{ \Ric_x(v)}{2(m+2)} \right|  \leq  C \left( \beta + \veps  +  \frac{\log(n)^{p_m}}{ n^{1/m}\veps^3}  \right),
 \end{equation}
 for all $x,y \in \X$ satisfying $3\delta_0 \leq \hat{d}_g(x,y) \leq \frac{1}{3} \delta_1$.  The quantities $\delta_0$, $\delta_1$, $v$ and $p_m$ are as in Theorem \ref{thm:consistency-non-asymp}. In case $m=1$, the above statement continues to hold after substituting $\frac{(\log(n))^{p_m}}{n^{1/m}}$ with $\frac{(\log(n))^{1/2}}{n^{1/2}}$.
 
 \end{theorem}


 \color{black}

\subsection{Global Curvature Lower Bounds}
\label{sec:MainGlobal}

We turn our attention to Problem~\ref{prob2} and state two theorems relating global lower bounds for $\Ric$ and $\kappa_G$. In section \ref{sec:Applications}, we complement our analysis with a discussion of applications of the curvature lower bounds stated below.

In our first result we assume access to the geodesic distance $d_\M$.
\begin{theorem}[Global lower bounds]
Let $\M$ be an $m$-dimensional, compact, boundaryless, connected, smooth manifold embedded in $\R^d$ with Ricci curvature lower bounded by $2(m+2) K$. Let $\X= \{ x_1, \dots, x_n \}$ consist of i.i.d. samples from the uniform distribution on $\M$. Let $w_\veps= w_{\veps, \M}$ be defined according to \eqref{eqn:WeightsManifold}, let $d_G= d_{G,\M} $ be the metric defined in \eqref{eq:DistancefromGeodesic} for a profile function $\psi$ satisfying Assumptions \ref{assumpPsi}, and let $\kappa_G$ be the Ollivier Ricci curvature  induced by these choices of Ollivier balls and metric (see \eqref{eqn:GraphCurvature}).

Under Assumption \ref{assump}, for every $s >1$ there is a constant $C$ such that, with probability at least $1- C n^{-s}$, we have
\begin{equation}
   \frac{\kappa_G(x,y)}{\veps^2} \geq  \min \left\{ s_K K  - C\left (\veps +  \frac{\log(n)^{p_m}}{ n^{1/m}\veps^3} \right), \frac{1}{2\veps^2} \right\}, \quad \forall x, y \in \X,
\label{eqn:LowerboundTheorem1}
\end{equation}
where the factor $s_K$ is given by
\[ s_K := \begin{cases} \frac{\psi'(0) c_0}{12c_1 C_\M}  & \text{ if } K \geq 0 \\ \frac{c_1}{c_0 \psi(0)}  &\text{ if } K <0, \end{cases} \]
where 
$c_0, c_1$ are as in \eqref{def:Delta_0Delta_1}, and $C_\M$ is a manifold dependent constant that in particular implies  $\frac{c_0}{12 c_1 C_\M} \leq 1$. Also, $p_m =3/4$ if $m=2$, while $p_m=1/m$ if $m\geq 3$. In case $m=1$, the above statement continues to hold after substituting $\frac{(\log(n))^{p_m}}{n^{1/m}}$ with $\frac{(\log(n))^{1/2}}{n^{1/2}}$.
\label{thm:GlobalBounds1}
\end{theorem} 
\nc

\begin{remark}
The rescaling of $\kappa_G$ by $\veps^2$ is the right scaling when passing to the continuum limit, i.e. $n \rightarrow \infty$ and $\veps\rightarrow 0$. This can already be seen from our pointwise consistency results in section \ref{sec:MainPointwiseState}, but it can also be interpreted as a way to properly rescale the time variable indexing the discrete-time random walk on $G$. In particular, we will see in section \ref{sec:LaplaciansandRegression} that \eqref{eqn:LowerboundTheorem1} implies novel contraction properties for the heat flow (continuous time) on $G$ when we assume the manifold $\M$ to be positively curved.
\end{remark}



\begin{remark}
   The factor $s_K$ makes the lower bound in \eqref{eqn:LowerboundTheorem1}  looser than the lower bound for $\M$'s Ricci curvature:  when $K\geq 0$, $s_K$ is necessarily smaller than one (but still strictly positive, since $\psi'(0)>0$), whereas when $K <0$ the quantity $s_K$ is greater than one. 
   The appearance of $s_K$ is due to the fact that in our analysis we must glue together two estimates that hold at different length-scales, and, in doing so, we end up with a suboptimal bound. Presumably, our analysis can be sharpened, but this aim is out of the scope of this paper.  
\end{remark}

A second result, which only assumes access to a sufficiently sharp data-driven approximation of the geodesic distance $d_\M$, is stated below.

\begin{theorem}[Consistency of global bounds 2]
\label{thm:GlobalBounds2}
Let $\M$ be an $m$-dimensional, compact, boundaryless, connected, smooth manifold embedded in $\R^d$ with Ricci curvature lower bounded by $2(m+2) K$. Let $\X= \{ x_1, \dots, x_n \}$ consist of i.i.d. samples from the uniform distribution on $\M$. Let $w_\veps= w_{\veps, \X}$ be defined according to \eqref{eqn:WeightsDataDriven}, for a data-driven approximation $\hat{d}_g$ of $d_\M$ satisfying Assumption \ref{assump:hatdg}. Let $d_G= d_{G,\X} $ be the metric defined in \eqref{eq:DistancefromDataDriven} for a profile function $\psi$ satisfying Assumptions \ref{assumpPsi}, and let $\kappa_G$ be the Ollivier Ricci curvature  induced by these choices of Ollivier balls and metric (see \eqref{eqn:GraphCurvature}).

Under Assumption \ref{assump}, for every $s >1$ there is a constant $C$ such that, with probability at least $1- C n^{-s} - C\exp(- \zeta(n, \beta,\veps))$, we have
\begin{equation}
   \frac{\kappa_G(x,y)}{\veps^2} \geq  \min \left\{ s_K K  - C\left (\veps + \beta+  \frac{\log(n)^{p_m}}{ n^{1/m}\veps^3}\right), \frac{1}{4\veps^2} \right\}, \quad \forall x, y \in \X,
\label{eqn:LowerboundTheorem2}
\end{equation}
where the factor $s_K$ is as in Theorem \ref{thm:GlobalBounds1}, and $p_m =3/4$ if $m=2$, while $p_m=1/m$ if $m\geq 3$. In case $m=1$, the above statement continues to hold after substituting $\frac{(\log(n))^{p_m}}{n^{1/m}}$ with $\frac{(\log(n))^{1/2}}{n^{1/2}}$.

\end{theorem}

As we will discuss in section \ref{sec:LaplaciansandRegression}, the above curvature lower bounds provide information on the behavior of Lipschitz seminorms along the heat flow induced by the graph Laplacian associated to the RGG $G=(\X, w_\veps)$.

\nc



 \subsection{Numerical illustration of results}
 \label{sec:MainNumerical}
We illustrate the recovery of global lower bounds (Theorem~\ref{thm:GlobalBounds1}) on the example of a unit $d$-sphere. Since the unit $d$-sphere has sectional curvature $1$, we expect to recover a global lower bound of $1$ for Ollivier's Ricci curvature in a random geometric graph, in the large-sample limit. To test this numerically, we sample $n$ points uniformly at random from the unit $d$-sphere, centered at the origin. Sampling is performed via sphere picking (also known as Muller-Marsaglia algorithm~\citep{muller}): We sample $d$ independent random variables from the standard normal distribution $z=(z_1, \dots, z_d)$. The point $\big(\sum_{i} z_i \big)^{-1/2} z$ lies on the unit $d$-sphere. We repeat this procedure $n$ times to generate a sample of size $n$. Figure~\ref{fig:sphere} shows the curvature distribution of the resulting random geometric graphs with different hyperparameters. We see that the Ricci curvatures of almost all edges concentrate around 1, which aligns well with our theoretical results\footnote{We suspect that the small number of outliers are due to numerical inaccuracies, specifically, (1) the sample distribution is not perfectly uniform and (2) some sample points do not lie exactly on the sphere.}.

In our numerical experiments, the 1-Wasserstein distance is computed via the Hungarian algorithm, which has a cubic complexity. Hence, for large sample sizes it is expensive to compute Ollivier's curvature on each edge in the RGG. However, as a byproduct of our global lower curvature bounds, we develop upper and lower bounds on the 1-Wasserstein distance, which do not require optimizing transport maps, but can be computed from combinatorial arguments. Thus, in applications in which we mainly rely on global lower curvature bounds (see examples in the next section), our approach nevertheless allows for an efficient characterization of the manifolds's geometry.

\begin{figure}[ht]
    \centering
\begin{subfigure}[t]{\linewidth}
\includegraphics[width=0.49\textwidth]{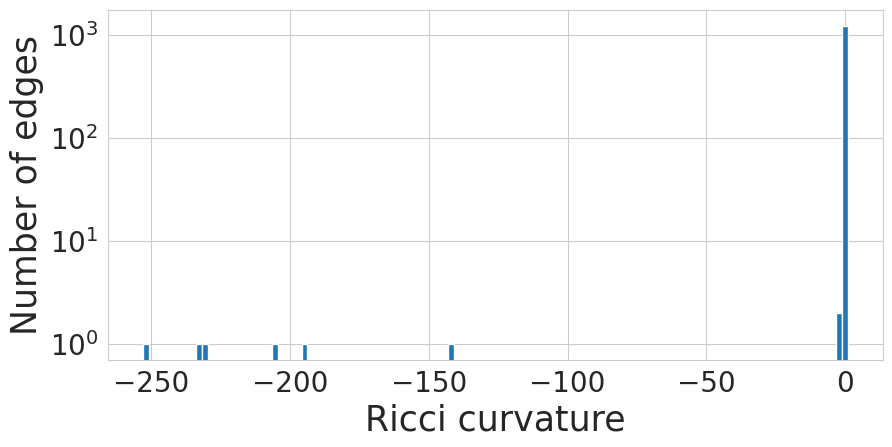}
\hfill
\includegraphics[width=0.49\textwidth]{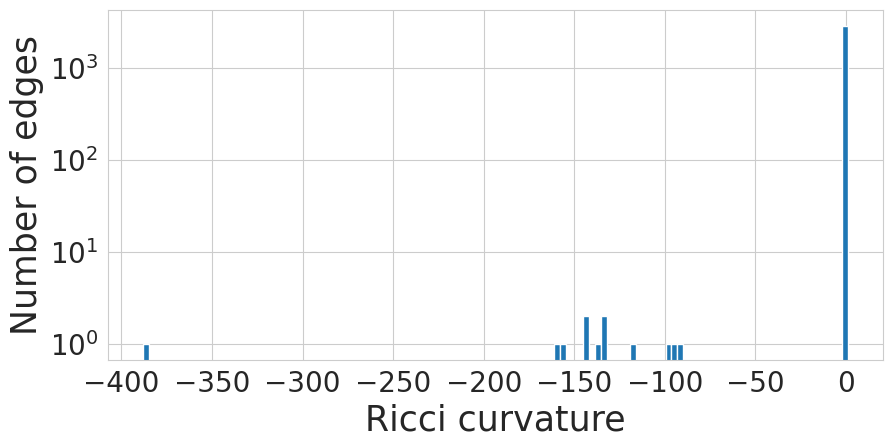}
\subcaption{Sample size $n=500$ (left) and $n=750$ (right) with parameters $\veps=0.2,\delta_0=0.01, \delta_1=0.02$.}
\end{subfigure}
  \\
  \begin{subfigure}[t]{\linewidth}
\includegraphics[width=0.49\textwidth]{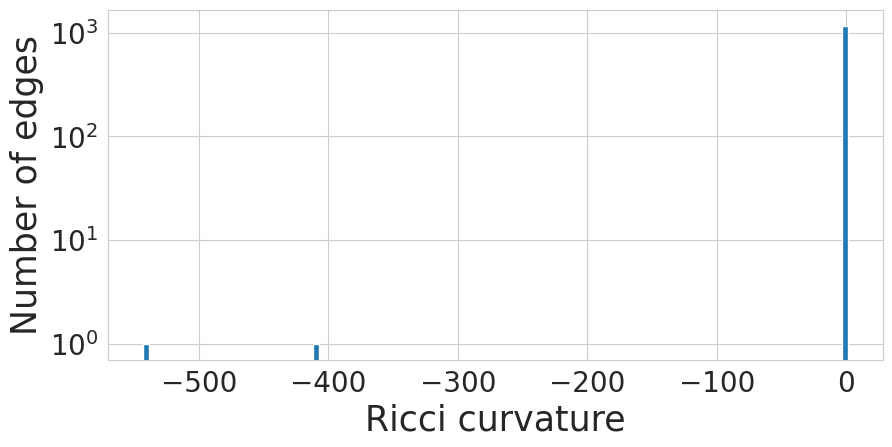}
\hfill
\includegraphics[width=0.49\textwidth]{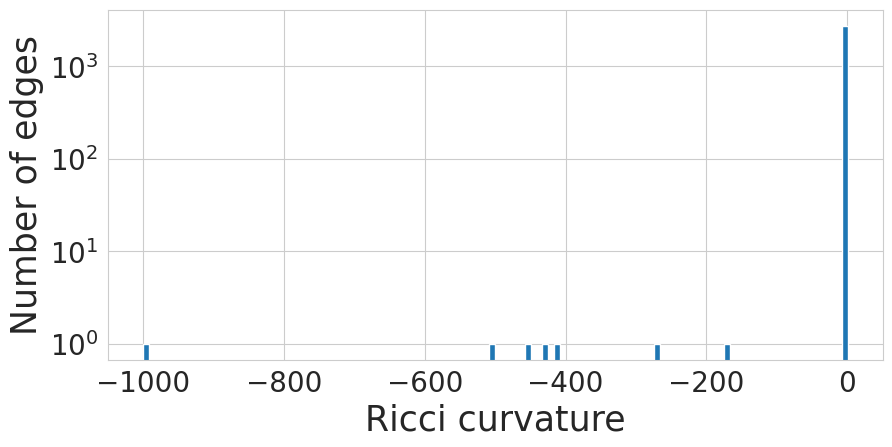}
\subcaption{Sample size $n=1000$ (left) and $n=1500$ (right) with parameters $\veps=0.1,\delta_0=0.005, \delta_1=0.01$}
\end{subfigure}
\caption{Distribution of curvature values for a random geometric graph $G_\veps$ constructed from 
$n$ data points sampled from the 2-sphere $\mathbb{S}^2 \subseteq \R^3$. The vertical axis is plotted in log-scale.}
  \label{fig:sphere}
\end{figure}
 \color{black}

 \section{Related Work}
 \label{sec:Literature}
 Throughout this work, we consider Ricci curvature in the sense of Ollivier~\citep{ollivier2009ricci,Ol}. Our analysis utilizes several theoretical results that date back to Ollivier's work~\citep{ollivier2009ricci}. 
 As mentioned above, the pointwise convergence of Ollivier's Ricci curvature on a RGGs toward a manifold's Ricci curvature has been previously studied in~\citep{van2021ollivier}, which gave \textit{asymptotic} guarantees assuming access to geodesic distances. Another result where access to geodesic distances was not assumed was also stated in that paper, but such result only holds in very special cases.
 
 Other popular discrete Ricci curvatures include notions by Forman~\citep{forman_bochners_2003} and Maas and Erbar~\citep{erbar_ricci_2012}, as well as a notion by Lin-Yau~\citep{lin-yau}, which is closely related to Ollivier's Ricci curvature. The relation between Forman's Ricci curvature and that of an underlying manifold has recently been explored in~\citep{meyer_approximate_2022}. Notably, Maas and Erbar's Ricci curvature allows for a log-Sobolev inequality (via a discrete Bakry-Emery theorem) and the inference of curvature lower bounds~\citep{fathi_entropic_2016}, although not in a tractable way that would allow one to infer its consistency in the RGG setting. Discrete Ricci curvatures have been related to a range of classical graph characteristics \citep{jost2014ollivier,weber2017characterizing,weber2018coarse} and have found applications in network analysis and machine learning~\citep{Farooq2017NetworkCA,ni2019community,tian2023mixed,fesser2023augmentations,sandhu2016ricci,weber2016forman,Ye2020Curvature}. Beyond \citep{van2021ollivier}, to the best of our knowledge there is no other work rigorously connecting the discrete Ricci curvature of a point cloud and the Ricci curvature of a manifold.
 
 Continuum limits of different geometric characteristics defined on data clouds have been explored in the literature. For example, the analysis of graph Laplacians and their convergence toward Laplace-Beltrami operators on manifolds has received a lot of attention in the last decades, e.g., in \citep{hein_graphs_2005,SINGER2006128} where pointwise consistency results are presented, and in \citep{burago2015graph,calder_improved_2020,CalderNGTLewicka,DUNSON2021282,Trillos2019_Error,WormelReich}, where spectral consistency is discussed. 

 Other works explore the discrete-to-continuum convergence of general data-driven distances, e.g., see \citep{howard2001geodesics, DiazetAl,DavisSethuraman, Hwang2016_Shortest,Bungert2022, little2022balancing} and some of the references therein. The papers \citep{DiazetAl,DavisSethuraman,Bungert2022}, for example, discuss convergence of distances defined on random geometric graphs (RGG), either in the i.i.d. setting or for Poisson point processes. The results from \citep{DavisSethuraman} are asymptotic, while the ones in \citep{DiazetAl,Bungert2022} provide high probability convergence rates in terms of an RGG's connectivity parameter. The results in \citep{Bungert2022}, for example, discuss the convergence of the ratio between certain expectations of distances at different scales. When combined with concentration inequalities, this allows the authors to prove rates of convergence, in sparse settings, for a semisupervised learning procedure known as Lipschitz learning; see \citep{Roith2022,Bungert2022}. Finally, \citep{HJB_Graphs} presents a graph-PDE approach to approximate geodesic distances by analyzing variants of the Eikonal equation on a graph.  Many of the approaches discussed in the aforementioned papers could potentially be used to define estimators $\hat{d}_g$ for $d_\M$. 

The approach for estimating $d_\M$ from data outlined in section \ref{sec:hatd_g} relies on the approximation of the second fundamental form from data. Some papers that explore the estimation of the second fundamental form are \citep{kim_curvature-aware_2013,aamari_nonasymptotic_2019,cao2021efficient} and are briefly discussed in Appendix \ref{sec:EstimatesSecondFundForm}.


\nc

 
  \section{Nonasymptotic Guarantees on Curvature Consistency}
\label{sec:NonAsymptoticGuarantees}
 
 In this section we present proofs for our main results. After some preliminary discussion, we first show the consistency of global curvature bounds, followed by a proof of pointwise, local consistency. A summary of the notation used throughout this section can be found in Table~\ref{tab:notation}, for the convenience of the reader.

\begin{table}[ht]
\centering
\begin{tabular}{ll}
      \toprule
      $B(x,\veps)$ & Euclidean ball \\
      $B_\M(x,\veps)$ & Ollivier ball, continuous setting \\
      $B_G(x,\veps)$ & Ollivier ball, discrete setting \\
      $d_\Mc(x,y)$ & geodesic distance on $\M$ \\
      $\hat{d}_g(x,y)$ & data-driven local approximation of $d_\M$ \\
       $\tilde d_{G,\M}(x,y)$ & pre-distance on $\X$, access to distances on $\M$ (Eq.~\ref{eq:DistancefromGeodesic}) \\
      $\tilde d_{G,\Xc}(x,y)$ & pre-distance on $\X$, data-driven (Eq.~\ref{eq:DistancefromDataDriven}) \\
      $d_{G,\M}(x,y)$ & distance on $\X$, access to distances on $\M$ (Eq.~\ref{eq:DistancefromGeodesic}) \\
      $d_{G,\Xc}(x,y)$ & distance on $\X$, data-driven (Eq.~\ref{eq:DistancefromDataDriven}) \\
      $K(u,v)$ & sectional curvature (Eq.~\ref{eq:sec-curv}) \\
      $\Ric_x(v)$ & Ricci curvature (Eq.~\ref{eq:ric-curv}) \\
      $W_1(\mu,\nu)$ & 1-Wasserstein distance w.r.t. $d_\M$ \\& or for a generic metric $d$ (clear from context) (Eq.~\ref{eq:W1-dist}) \\
      $W_{1,G}(\mu,\nu)$ & 1-Wasserstein distance w.r.t. $d_G$ (Eq.~\ref{eq:W1G-dist}) \\
      $W_\infty(\mu,\nu)$ & $\infty$-OT distance (Eq.~\ref{eq:W-inf}) \\
       \bottomrule
    \end{tabular}
    \caption{Summary of notation.}
    \label{tab:notation}
\end{table}

 \subsection{Preliminaries}
 
 

In this subsection we collect a series of preliminary results and estimates that we use in the proofs of our main results. 

\subsubsection{Some lemmas from optimal transport theory}

We recall the Kantorovich-Rubinstein duality theorem for the 1-Wasserstein metric between two probability measures over the same Polish metric space.

\begin{theorem}[Kantorovich-Rubinstein, cf. Thm~1.14,~\citep{VillaniBook}]
\label{thm:KantorovichRubinstein}
Let $\mu_1, \mu_2$ be two (Borel) probability measures over a Polish metric space $(\mathcal{U},d)$. Then
\[ W_1(\mu_1, \mu_2) := \sup_{ f \text{ s.t. } \mathrm{Lip}(f) \leq 1 } \int f(\tilde x ) d \mu_1(\tilde x) - \int f(\tilde y) d\mu_2(\tilde y) . \]
In the above, $\mathrm{Lip}(f)$ stands for the Lipschitz constant (relative to the metric $d$) of the function $f$.
\end{theorem}

Next, we recall the notion of \textit{glueing} of couplings. Given finite positive measures $\mu_1,\dots, \mu_k$ over a Polish space $(\mathcal{U},d)$, all of which have the same total mass, and given couplings $\pi_{12}, \pi_{23}, \dots, \pi_{k-1, k}$ with $\pi_{l, l+1} \in \Gamma(\mu_l, \mu_{l+1})$, we define $\Pi$, the glueing of the couplings $\pi_{l, l+1}$, as the measure over $\mathcal{U}^k$ satisfying
\begin{align}
\begin{split}
&\int \varphi(x_1, \dots, x_k) d\Pi(x_1, \dots, x_k) = 
\\ & \quad \int\int \dots \int \varphi(x_1, \dots, x_k) d\pi_{k-1, k}( x_k | x_{k-1} ) \dots d\pi_{1,2}(x_2| x_1) d\mu_1(x_1)  
\end{split}
\label{eq:Glueing}
\end{align}
for all regular enough test functions $\varphi$; in the above, $\pi_{l, l+1 }(\cdot| x_{l} )$ must be interpreted as the conditional distribution of $x_{l+1}$ given $x_l$ when $(x_{l},x_{l+1})$ are jointly distributed according to $\pi_{l, l+1}$. For given $l,s \in \{ 1, \dots, k \}$ consider the map $T_{l,s}: (x_1, \dots, x_k) \mapsto (x_s, x_l)$. It is straightforward to see that $T_{l,s\sharp } \Pi \in \Gamma(\mu_s, \mu_l)$.

Next, we present the following lemma.

\begin{lemma}
\label{lem:SplittingMass}
 Let $\mu_1, \mu_2, \tilde \mu_1,\tilde \mu _2 $ be finite positive measures defined over the same Polish space $(\mathcal{U},d)$, satisfying $\mu_1(\mathcal{U})= \tilde \mu_{1}(\mathcal{U})$ and $\mu_2(\mathcal{U})= \tilde \mu_{2}(\mathcal{U})$. Then
 \[  W_1( \mu_1 +\mu_2 , \tilde \mu_1 + \tilde \mu_2  ) \leq  W_1( \mu_1,  \tilde \mu_1  )  +  W_1( \mu_2 , \tilde \mu_2  ).  \]
\end{lemma}
\begin{proof}
The desired inequality follows from the observation that for any two couplings $\pi_1 \in \Gamma(\mu_1, \tilde \mu_1)$ and $\pi_2 \in \Gamma(\mu_2, \tilde \mu_2)$ we have $\pi_1 + \pi_2 \in \Gamma(\mu_1+ \mu_2 , \tilde \mu_1 + \tilde \mu_2)$.
\end{proof}

\subsubsection{Some estimates for the $\infty$-OT distance between measures}

In the proofs of our main results we will make use of the $\infty$-OT distance $W_\infty(\cdot,\cdot)$ between probability measures defined over the same metric space. Precisely, let $\mu_1, \mu_2$ be two (Borel) probability measures over a Polish metric space $(\mathcal{U},d)$. We define $W_\infty(\mu_1, \mu_2)$ as
\begin{equation}\label{eq:W-inf}
     W_\infty(\mu_1, \mu_2) := \inf_{\pi \in \Gamma(\mu_1, \mu_2)} \sup_{(\tilde x , \tilde y) \in \supp(\pi)  }  d(\tilde x , \tilde y),
\end{equation}
where $\supp(\pi)$ stands for the support of the measure $\pi$.

The following results relate, on the one hand, the $\infty$-OT distance between two measures with densities with respect to the uniform measure over a Euclidean (or geodesic) ball, and on the other hand the $L^\infty$ distance between the densities.

\begin{proposition}[cf. Thm.~1.2,~\citep{InftyOT}]
Let $\mu_1, \mu_2$ be two probability measures over $B(0,\veps) \subseteq \R^m$ with densities $\rho_1, \rho_2$ with respect to the uniform probability measure over $B(0,\veps)$ satisfying:
\[ \frac{1}{\alpha} \leq \rho_1(x), \rho_2(x) \leq \alpha,  \]
for some $\alpha>1$. Then 
\[ W_\infty(\mu_1, \mu_2) \leq \alpha C_m \veps  \lVert  \rho_1 - \rho_2  \rVert_{L^\infty(B(0,\veps))}, \]
where $C_m$ only depends on dimension $m$ and not on $\alpha$, $\veps$, or $\rho_1, \rho_2$.
\label{prop:InftyOT}
\end{proposition}
\begin{proof}
Theorem 1.2. in \citep{InftyOT} gives the result for $\veps=1$. The general case follows immediately from a rescaling argument.  
\end{proof}

\begin{corollary}
Let $\M$ be a smooth, compact Riemannian manifold without boundary and let $x \in \M$. Let $\veps < \iota_\M/2$. Let $\mu_1, \mu_2$ be two probability measures over $B_\M(x,\veps)$ with densities $\rho_1, \rho_2$, with respect to the uniform probability measure over $B_\M(x,\veps)$, that satisfy:
\[ \frac{1}{\alpha} \leq \rho_1(x), \rho_2(x) \leq \alpha,  \]
for some $\alpha>1$. Then 
\[ W_\infty(\mu_1, \mu_2) \leq \alpha C_\M \veps  \lVert  \rho_1 - \rho_2  \rVert_{L^\infty(B_\M(x,\veps))}, \]
where $C_\M$ is a manifold dependent constant that does not depend on $x$, $\alpha$, $\veps$, or $\rho_1, \rho_2$.
\label{cor:InftyOT}
\end{corollary}
\begin{proof}
   Since for every $x\in \M$ the map $\exp_x : B(0,\veps) \rightarrow B_\M(x,\veps)$ is bi-Lipschitz, with bi-Lipschitz constants uniformly bounded over all $0 < \veps < \iota_\M/2 $ and all $x \in \M$, the desired inequality follows directly from Proposition \ref{prop:InftyOT}.   
\end{proof}

We also discuss probabilistic bounds for the $\infty$-OT distance between $\mu$ and the empirical measure $\mu_n$. Specifically, we will use the following result that can be found in \citep{Trillos2019_Error} (see also references therein).

\begin{theorem}[Theorem 2 in \citep{Trillos2019_Error} for $m \geq 2$ and Theorem 1 in \citep{Liu2019} for $m=1$]  
\label{thm:inftyOTEmpiricalBound}
Let $\mu_n$ be the empirical measure of $n$ i.i.d. samples from $\mu$. Then in case $m\geq 2$ (recall $m$ is $\M$'s dimension), for any $s>1$ and $n \in \N$, we have
\begin{equation}
  W_\infty(\mu, \mu_n)  = \min_{T \: : \: T_{\sharp} \mu =\mu_n} \sup_{x \in \M} d_\M( T(x), x) \leq A_{\M, s} \frac{(\log(n))^{p_m}}{n^{1/m}} , 
   \label{eq:InftyOTMap}
\end{equation}
with probability at least $1- C_{\M, s} n^{-s}$, where $p_m$ is a dimension dependent power: $p_m= 3/4$ when $m=2$, and $p_m= 1/m$ when $m\geq 3$. The constants $C_{\M, s}$ and $A_{\M, s}$ only depend on $s $ and on $\M$. In the sequel, we use $T_n$ to denote a minimizer in the above formula. In case $m=1$, the above statement continues to hold after substituting $\frac{(\log(n))^{p_m}}{n^{1/m}}$ with $\frac{(\log(n))^{1/2}}{n^{1/2}}$.




\end{theorem}

\subsubsection{Proof of Theorem \ref{thm:ollivier1}}
\label{sec:ProofTheorem1}
We now provide a proof of Theorem~\ref{thm:ollivier1}, which we restate below for convenience.
\begin{theorem}[Ollivier~\citep{ollivier2009ricci}]\label{thm:ollivier}
\begin{equation}
    \Big\vert \k_\Mc (x,y) -  \frac{\veps^2}{2(m+2)} \Ric_x(v)  \Big\vert \leq  \left(  C \veps^2 d_\Mc(x,y) + C' \veps^3
    \right)\; ,
\end{equation}
where $y$ is a point on the geodesic from $x$ in direction $v\in T_x\M$ and $\veps$ is the radius of Ollivier balls. 
\end{theorem}

\begin{proof}[Proof of Theorem \ref{thm:ollivier}]
Let $x,y \in \M$ and let $\mathcal{P}: B_\M(x,\veps) \rightarrow B_\M(y, \veps)$ be the map from \eqref{eqn:ParallelF}. Then, according to Proposition 6 in \citep{ollivier2009ricci}, we have

 \begin{equation}
 d_\M(\tilde x , \mathcal{P}( \tilde x )) = d_\M(x,y) \left( 1 - d_\M(x, \tilde x)^2\left(\frac{K(v,w)}{2}  + O( d_\M(x,y) + \veps) \right) \right),
 \label{eqn:DistLeviCivita}
 \end{equation}
where $v= \frac{\log_x(y)}{|\log_x(y)|}$, $w=\log_x(\tilde x)$, and $K(v,w)$ is the sectional curvature in the plane generated by the vectors $v,w \in T_x\M$. The distance between $\tilde x$ and $\tilde y:= \mathcal{P}(\tilde x)$ is almost equal to the distance between $x$ and $y$, and the correction term of order 3 is precisely captured by the sectional curvature between vectors $v$ and $w$.

The measure $\mathcal{P}_{\sharp} \mu_x^\M$, although not exactly equal to $ \mu_y^\M$, has a density $\rho_{xy}$ with respect to $\mu_y^\M$ that satisfies
\[  \sup_{\tilde y \in B_\M(y,\veps) }  | \rho_{xy} (\tilde y) - 1 | \leq C_\M  d_\M(x,y) \veps^2 ,   \]
as follows from the discussion in the proof of Proposition 6 in section 8 in \citep{ollivier2009ricci}. Combining the above estimate with Corollary \ref{cor:InftyOT} we get 
\[  W_\infty( \mathcal{P}_{\sharp} \mu_x^\M , \mu_y^\M ) \leq C_\M \veps \sup_{\tilde y \in B_\M(y,\veps) }  | \rho_{xy} (\tilde y) - 1 |  \leq C_\M d_\M(x,y) \veps^3.\] 
We can thus find a map $T_y: B_\M(y, \veps) \rightarrow B_\M(y,\veps)$ such that $T_{y \sharp} ( {\mathcal{P}}_{\sharp} \mu_x^\M) = \mu_y^\M $ and such that 
\begin{equation}
\label{eq:AuxT_y}
   \sup_{y' \in B_\M(y, \veps )} | y' - T_y(y') | \leq C_\M' d_\M(x,y) \veps^3; 
\end{equation} 
see \citep{ChampionExistenceOTMaps}. If we now define the function $\F: B_\M(x,\veps) \rightarrow B_\M(y,\veps)$ as
\begin{equation}
    \F:= T_y \circ \mathcal{P},
    \label{eqn:DefG}
\end{equation}
we see that 
\begin{equation*}
\F_{\sharp} \mu_x^\M = \mu_y^\M.
\end{equation*}
Moreover, for every $\tilde x \in B_\M(x, \veps)$ we have
 \begin{equation}
 d_\M(\tilde x , \F( \tilde x )) = d_\M(x,y) \left( 1 - d_\M(x, \tilde x)^2\left(\frac{K(v,w)}{2}  + O( d_\M(x,y) + \veps) \right) \right),
 \label{eqn:DistanceH}
 \end{equation}
where we recall $v= \frac{\log_x(y)}{|\log_x(y)|}$ and $w=\log_x(\tilde x)$. It follows that
  \begin{align}
  \begin{split}
    W_{1}(\mu_x^\M, \mu_y^\M)  & \leq \int_{B_\M(x, \veps)} d_\M(\tilde x , \F(\tilde x)) d \mu_x^\M(\tilde x) 
    \\& \leq d_\M(x, y ) -\veps^2 d_\M(x,y) \frac{\Ric_x(v)}{2(m+2)}  + O( d_\M(x,y)^2 \veps^2 + d_\M(x,y) \veps^3 ), 
    \end{split}
    \label{eq:AuxTheorem1}
  \end{align} 
and in turn
\[ \frac{1}{\veps^2}\kappa_\M(x,y)  + O(d_\M(x,y)  + \veps) \geq  \frac{\Ric_x(v)}{2(m+2)},   \]
 giving a lower bound for $\kappa_\M$.

To obtain a matching upper bound, we follow \citep{ollivier2009ricci} and construct a function $f: \M \rightarrow \R$ that is $1$-Lipschitz with respect to $d_\M(\cdot,\cdot )$ and that almost realizes the sup in the Kantorovich-Rubinstein dual formulation of the 1-Wasserstein distance between $\mu_x^\M$ and $\mu^\M_y$ (see Theorem \ref{thm:KantorovichRubinstein}). To define this function, let us consider $0<r_0< \iota_\M$, and suppose that $\veps$ is small enough and $x,y$ are sufficiently close so that $B_\M(x,\veps), B_\M(y, \veps) \subseteq  B_\M(x, r_0/4) $. Let $E_0:= \{ v' \in T_x \M \: : \:  \langle v' , \log_x(y) \rangle=0  \}$. We first define $f : B_\M(x,r_0) \rightarrow \R$ by 
\begin{equation}
f(z):= \begin{cases}
\dist(z, \exp_x(E_0)) & \text{ if } \langle \log_x(z) , \log_x(y) \rangle \geq 0 \\
- \dist(z,\exp_x(E_0)) & \text{ if } \langle \log_x(z) , \log_x(y) \rangle < 0, 
\end{cases}
\label{eqn:functionF}
\end{equation}
which is $1$-Lipschitz in its domain, and then extend it to a global $1$-Lipschitz function using the {McShane-Whitney extension theorem}. Following the steps in section 8 in \citep{ollivier2009ricci} we can then see that
 \begin{equation}
 f( \F (\tilde x)  ) - f(\tilde x)=  d_\M(x,y) \left( 1 - d_\M(x, \tilde x )^2\left(\frac{K(v,w)}{2}  + O( d_\M(x,y) + \veps) \right) \right)
 \end{equation}
  for every $\tilde x \in B_\M(x,\veps) $, where we recall $v= \frac{\log_x(y)}{|\log_x(y)|}$, $w=\log_x(\tilde x)$ and $\F$ is as in \eqref{eqn:DefG}. Integrating with respect to $\mu_x^\M$ and using the Kantorovich-Rubinstein theorem we get
\begin{align}
\begin{split}
  W_1(\mu_x^\M, \mu_y^\M)   & \geq \int_\M f(\tilde y) d\mu^\M_y(\tilde y) -  \int_\M f(\tilde x) d\mu^\M_x(\tilde x)
  \\& = \int_\M f(\F(\tilde x)) d\mu^\M_x(\tilde x) -  \int_\M f(\tilde x) d\mu^\M_x(\tilde x)
  \\& =  d_\M(x, y ) -\veps^2 d_\M(x,y) \frac{\Ric_x(v)}{2(m+2)}  + O( d_\M(x,y)^2 \veps^2 + d_\M(x,y) \veps^3 ),
  \end{split}
  \label{eq:FormulafRic}
\end{align}
from where we can now obtain
\[ \frac{1}{\veps^2}\kappa_\M(x,y)   \leq  \frac{\Ric_x(v)}{2(m+2)} + O(d_\M(x,y)  + \veps).   \]
\end{proof}


\nc

\subsubsection{Some additional lemmas}

In this section we collect a few lemmas that we use in the proof of our main results.

\begin{lemma}
\label{lem:add1}
There is a constant $c$ such that for all small enough $\veps_0>0$ and all $x\in \M$ we have
\[ v_m (1 - c\veps_0^2) \veps_0^m  \leq  \vol( B_\M(x, \veps_0) ) \leq  v_m (1 + c\veps_0^2) \veps_0^m,  \]
where $v_m$ is the volume of the $m$-dimensional Euclidean ball.

Moreover, if $\veps_0>0$ is such that $W_\infty(\mu, \mu_n) \leq \frac{1}{2} \veps_0$, then
\[  \mu(B_\M(x, \veps_0 - W_\infty(\mu, \mu_n) )) \leq   \mu_n(B_\M(x,\veps_0)) \leq \mu(B_\M(x,\veps_0+ W_\infty(\mu, \mu_n))).  \]

\end{lemma}
\begin{proof}
The first part is a standard result in differential geometry (see for example 1.35 in \citep{Trillos2019_Error}). The second part is immediate from the definition of $W_\infty(\mu, \mu_n)$.
\end{proof}




Given the assumed compactness and smoothness of the manifold $\M$, it is straightforward to show that there exists a constant $C_\M \geq 1$ such that
\begin{equation}
\label{eq:BallDifference}
 \mu_x^\M( B_\M(x,\veps) \setminus (B_\M(x,\veps) \cap B_\M(y,\veps) ) ) \leq C_\M  \frac{d_\M(x,y)}{ \veps }
 \end{equation}
for all $x, y \in \M$ and all $\veps \leq \iota_\M/2$; indeed, this type of estimate is easily proved in Euclidean space and can be extended to the manifold setting for all small enough $\veps$ using coarse bounds on the metric distortion by the exponential map around a given point on the manifold. With the aid of standard concentration inequalities we can get a similar estimate to \eqref{eq:BallDifference} when $x \in \X$ and $\mu_x^\M $ is replaced with the empirical measure $\mu_x^G$. This is the content of the next lemma.

\begin{lemma} 
Provided that $\frac{W_\infty(\mu, \mu_n)}{\veps} $ is sufficiently small, we have
\[  \frac{\mu_n(B_\M(x,\veps) \setminus B_\M(x,\veps) \cap B_\M(y,\veps) )}{\mu_n(B_\M(x,\veps))} \leq  \frac{\psi(0) c_0}{6}  \]
for all $x,y\in \X$ satisfying $0<d_\M(x,y) \leq  \frac{\psi(0)}{12 C_\M} \delta_0$.
\label{lem:OverlapBalls2}
\end{lemma}
\begin{proof}
This result follows from \eqref{eq:BallDifference}, Lemma \ref{lem:add1}, and the smallness assumption of $W_\infty(\mu, \mu_n)$ relative to $\veps$.
\end{proof}



\nc

 \subsection{Proofs of global curvature bounds}


%

We start by proving Theorem \ref{thm:GlobalBounds1}.

\begin{proof}[Proof of Theorem \ref{thm:GlobalBounds1}]

Thanks to Lemma \ref{lem:CurvatureGeodesicProp} it is enough to prove the lower bound under the assumption that $x,y\in \X$ are two distinct points such that  $d_{G,\M}(x,y) = \tilde d_{G,\M}(x,y) \leq \delta_1$. Notice that $\tilde d_{G, \M}(x,y) = \delta_0 \psi( d_\M(x,y)/\delta_0)$ and thus we may further split the analysis into different cases determined by the value of $d_\M(x,y)$. It is worth recalling that in the setting considered here we have $B_G(x, \veps)= B_\M(x, \veps) \cap \X$ and $B_G(y, \veps)= B_\M(y, \veps) \cap \X$.


\textbf{Case 1:} $0<d_\M(x,y) \leq  \frac{\psi(0)}{12 C_\M} \delta_0$, where $C_\M$ is as in \eqref{eq:BallDifference}.  

We may assume without the loss of generality that $| B_G(x,\veps)| \geq | B_G(y,\veps)| $, for otherwise we can swap the roles of $x$ and $y$. From Lemma \ref{lem:OverlapBalls2} it follows 
\[\mu_x^G( B_\M(x,\veps) \setminus (B_\M(x,\veps) \cap B_\M(y,\veps) ) )  \leq \frac{\psi(0) c_0}{6}.\] 
Also, for all $\tilde x \in B_G(x,\veps) $ and $\tilde y \in B_G(y, \veps)$ we have
\[  d_{G, \M}(\tilde x, \tilde y ) \leq d_{G, \M}( x, \tilde x )  + d_{G, \M}( \tilde x, \tilde y )  +  d_{G, \M}( \tilde y, y )  \leq 2 \delta_0 \psi(\veps /\delta_0)  + \delta_0 \leq 3\veps.  \]
By selecting a coupling between $\mu_x^G$ and $\mu_y^G$ that leaves all mass of $\mu_x^G$ in $B_\M(x,\veps) \cap B_\M(y,\veps)$ fixed, the above estimates imply
\[ W_{1,G}(\mu_x^G, \mu_y^G) \leq   3 \veps  \mu_x^G( B_\M(x, \veps) \setminus (B_\M(x, \veps) \cap B_\M(y, \veps) ) ) \leq  \frac{1}{2}\psi(0) \delta_0 . \]
In addition, since by definition we have $d_{G,\M}(x,y) = \tilde{d}_{G, \M}(x,y) \geq \delta_0 \psi(0)$, it follows
\[ \frac{\kappa_G(x,y)}{\veps^2} = \frac{1}{\veps^2} \left(1- \frac{W_{1,G}(\mu_x^G, \mu_y^G)}{d_{G, \M}(x,y)}\right) \geq \frac{1}{2\veps^2}. \]

\textbf{Case 2:} $\frac{\psi(0)}{12 C_\M} \delta_0 \leq d_\M(x,y) \leq \delta_1 - 2 \veps $. 

We start by finding a good upper bound for $W_{1,G}(\mu_x^G, \mu_y^G )$. Without the loss of generality we can assume that 
\[ a:= \frac{\mu( B_\M(x,\veps))}{ \mu_n( B_G(x,\veps) ) } \leq \frac{\mu( B_\M(y,\veps))}{ \mu_n( B_G(y,\veps) ) }   ,\]
for otherwise we can swap the roles of $x$ and $y$.
We split the measure $\mu_{x}^G$ into
\[ \mu_x^G = \mu_x^G\lfloor_{B_\M(x, \veps')} +  \mu_x^G \lfloor_{B_\M(x,\veps) \setminus B_\M(x,\veps')}  ,\]
where $\veps' := \veps - 3W_\infty(\mu, \mu_n)- C_\M' \veps^4$ and where the measures on the right hand side represent the restrictions of $\mu_x^G$ to $B_\M(x,\veps')$ and $B_\M(x,\veps ) \setminus B_\M(x, \veps')$, respectively. We decompose the measure $\mu_y^G$ as 
\[ \mu_y^G= \mu_{y,1}^G + \mu_{y,2}^G \]
for two positive measures $\mu_{y,1}^G$ and $\mu_{y,2}^G$ that we define below, the first of which will be suitably coupled with $\mu_x^G\lfloor_{B_\M(x, \veps')} $ while the second one will be coupled with $\mu_x^G \lfloor_{B_\M(x,\veps) \setminus B_\M(x,\veps')} $.

Precisely, the measure $\mu_{y,1}^G$ is defined as
 \[ \mu_{y,1}^G :=  a T_{n \sharp} ( \F_{\sharp }  (  \mu_x^\M \lfloor_{T_n^{-1}(B_\M(x,\veps'))} ) ),     \]
where $T_n : \M \rightarrow \X$ is an $\infty$-OT map between $\mu$ and $\mu_n$ as defined in Theorem \ref{thm:inftyOTEmpiricalBound} and $\F$ is the map defined in \eqref{eqn:DefG}. We will show that $\mu_{y,1}^G \leq \mu^G_y $, which would allow us to take $\mu_{y,2}^G := \mu_y^G - \mu_{y,1}^G$. To see that indeed $\mu_{y,1}^G \leq \mu^G_y $, we first observe that $T_n^{-1}( B_\M(x,\veps') )$ is contained in $B_\M(x, \veps - 2 W_\infty(\mu, \mu_n) - C_\M' \veps^4 )$. From \eqref{eq:AuxT_y} and ii) in Assumption \ref{assump} it follows that $\F (T_n^{-1}( B_\M(x,\veps') )) \subseteq B_\M(y, \veps - 2W_\infty(\mu, \mu_n) ) $. Finally, $T_n(\F (T_n^{-1}( B_\M(x,\veps') ))) \subseteq B_\M(y, \veps- W_\infty(\mu, \mu_n)) $. From this we see that the support of $\mu_{y,1}^G$ is contained in $B_\M(y, \veps- W_\infty(\mu, \mu_n))$. Now, let $A\subseteq B_\M(y, \veps- W_\infty(\mu, \mu_n)) $.  We see that
\begin{align*}
   \mu_{y,1}^G(A) &= a \mu_x^\M \left( T_n^{-1}( B_\M(x,\veps') ) \cap \F^{-1}( T_n^{-1} (A)) \right) 
   \\& \leq a \mu_x^\M( \F^{-1}(T_n^{-1}(A)))
   \\& = a\mu_y^\M( T_n^{-1}(A) )
   \\& = \frac{a}{\mu(B_\M(y,\veps))} \mu( T_n^{-1}(A) ) 
   \\&= \frac{a}{\mu(B_\M(y,\veps))} \mu_n( A )
   \\&= \frac{a \mu_n(B_G(y,\veps))}{\mu(B_\M(y,\veps))} \mu_y^G( A )
   \\& \leq \mu_y^G(A).
\end{align*}
In the above, the second equality follows from the fact that $T_n^{-1}(A) \subseteq B_\M(y,\veps)$ and the fact that $\F_{\sharp} \mu_x^\M = \mu_y^\M$; the fourth equality follows from the fact that $T_{n \sharp} \mu = \mu_n$; the last inequality follows from the definition of $a$. Since $A$ was arbitrary, we conclude that indeed $\mu_{y,1}^G \leq \mu_{y}^G$.

Next, we show that $\mu_x^G\lfloor_{B_\M(x, \veps')}$ and $\mu_{y,1}^G$ have the same total mass and then construct a suitable coupling between them. Indeed, on one hand we have $\mu_x^G\lfloor_{B_\M(x, \veps')}(\X)= \mu_x^G ( B_\M(x, \veps') ) = \frac{\mu_n( B_\M(x, \veps') )}{\mu_n( B_G(x, \veps) )}$. On the other hand, 
\begin{align*}
 \begin{split}
  \mu_{y,1}^G(\X) & = a \mu_x^\M( T_n^{-1}(B_\M(x, \veps'))) = \frac{a}{\mu( B_\M(x, \veps) )} \mu ( T_n^{-1}(B_\M(x,\veps'))) 
  \\& = \frac{a}{\mu( B_\M(x, \veps) )} \mu_n ( B_\M(x,\veps'))   =  \frac{\mu_n( B_\M(x, \veps') )}{\mu_n( B_G(x, \veps) )},   
 \end{split}   
\end{align*}
which implies that the measures indeed have the same total mass. To construct a suitable coupling $\pi_1^G \in \Gamma( \mu_x^G\lfloor_{B_\M(x, \veps')}  , \mu_{y,1}^G) $, we first introduce the measure
\[ \tilde{\nu}_1 := \frac{a}{ \mu(B_\M(x,\veps))} \mu \lfloor_{ T_n^{-1}(B_\M(x,\veps') ) }. \]
Observe that $\tilde{\pi}_1:= ( T_n \times Id  )_{\sharp} \tilde{\nu}_1  $
belongs to $\Gamma( \mu_x^G\lfloor_{B_\M(x, \veps')}  , \tilde \nu_1)$ and 
\[  d_\M(\tilde x,\tilde x') \leq W_\infty(\mu, \mu_n), \quad \forall (\tilde x , \tilde x') \in \supp(\tilde \pi_1).   \]
Also, $\tilde{\pi}_2:= (Id \times \F)_{\sharp} \tilde{\nu}_1 \in \Gamma(\tilde{\nu}_1, \F_{\sharp} \tilde{\nu}_1)$ satisfies
\[ d_\M(\tilde x' , \tilde y') = d_\M(x,y) \left( 1 - d_\M(x, \tilde x')^2\left(\frac{K(v,w')}{2}  + O( d_\M(x,y) + \veps) \right) \right)\]
for all $(\tilde x', \tilde y ') \in \supp(\tilde{\pi}_2)$, according to \eqref{eqn:DistanceH}; in the above, $v= \frac{\log_x(y)}{|\log_x(y)|}$ and $w'=\log_x(\tilde x')$. Finally,  $\tilde{\pi}_3 := (Id \times T_n)_{\sharp}( \F_{\sharp} \tilde{\nu}_1 ) \in \Gamma( \F_{\sharp } \tilde{\nu}_1, \mu_{y,1}^G  ) $ satisfies
\[  d_\M(\tilde y',\tilde y) \leq W_\infty(\mu, \mu_n), \quad \forall (\tilde y ' , \tilde y) \in \supp(\tilde \pi_3).  \]
We can then define $\pi_1^G \in \Gamma( \mu_x^G\lfloor_{B_\M(x, \veps')}  , \mu_{y,1}^G) $ as
\[  \pi_{1}^G := T_{1,4 \sharp} \Pi. \]
where $\Pi$ is the glueing of the couplings $\tilde \pi_1, \tilde \pi_2, \tilde \pi_3$ as defined in \eqref{eq:Glueing} and $T_{1,4}$ is the projection onto the first and fourth coordinates introduced when we defined the glueing of couplings.  

We now proceed to estimate $W_{1,G}(\mu_x^G, \mu_y^G)$ from above using the coupling $\pi_{1}^G$. First, let $(\tilde x,\tilde x', \tilde y', \tilde y) \in \supp(\Pi) $. From the above discussion we have
\[ d_\M(\tilde x, \tilde y) = d_\M(x,y) \left( 1 - d_\M(x,\tilde x ')^2\left(\frac{K(v,w')}{2}  + O( d_\M(x,y) + \veps) \right) \right) + O(W_\infty(\mu, \mu_n)).   \] In particular, given the smallness of $W_\infty(\mu, \mu_n)$ relative to $\veps$ and the fact that $d_\M(x, y) \leq \delta_1 - 2 \veps$ we can assume without the loss of generality that $d_\M(\tilde x , \tilde y) \leq \delta_1$ and thus 
\[ d_{ G, \M}(\tilde x , \tilde y) \leq   \tilde d_{ G, \M}(\tilde x , \tilde y)  = \delta_0 \psi( d_\M(\tilde x,\tilde y)/\delta_0). \]
Using the fact that $\psi$ is non-decreasing combined with a simple Taylor expansion of $\psi$ around $d_\M( x, y)/\delta_0 $, we can bound the right hand side of the above by 
\begin{align*}
  \begin{split}
    \delta_0 \psi\left( \frac{d_\M(x , y)}{\delta_0} \right) & + \psi'\left( \frac{d_\M( x , y)}{\delta_0}\right) \left(   d_\M(\tilde x , \tilde y) - d_\M(x,y) \right)  
    \\ & + \frac{1}{2\delta_0} \lVert \psi'' \rVert_\infty (   d_\M(\tilde x , \tilde y) - d_\M(x,y) )^2
    \\ \leq  \delta_0 \psi\left( \frac{d_\M(x , y)}{\delta_0} \right) & - \frac{1}{2}\psi'\left( \frac{d_\M(x,y)}{\delta_0}\right) d_\M(x,y) \veps^2 K(v, w')
    \\& + O(\veps^4 + W_\infty(\mu, \mu_n));
  \end{split}  
\end{align*}
notice that $\lVert \psi'' \rVert_\infty$, the supremum norm of the second derivative of $\psi$, is finite by Assumption \ref{assumpPsi}; notice also that this second order correction term is of order $O( W_\infty(\mu, \mu_n)^2/\veps + \veps^5)$, which is much smaller than $O(\veps^4 + W_\infty(\mu, \mu_n))$. 

From the above estimates we get
\begin{align}
\begin{split}
  & W_{1,G}(\mu_x^G \lfloor_{B_\M(x, \veps')}, \mu_{y,1}^G )  \leq \int d_{G,\M}(\tilde x , \tilde y) d \pi_1^G(\tilde x , \tilde y)
 =  \int d_{G,\M}(\tilde x, \tilde y) d \Pi(\tilde x ,\tilde x', \tilde y', \tilde y)
\\& \leq \delta_0\psi(d_\M(x,y)/\delta_0) - \frac{1}{2} \psi'\left( \frac{d_\M(x,y)}{\delta_0} \right) d_\M(x,y)  \int d_\M(x, \tilde x')^2 K(v, \log_x(\tilde x')) d \tilde{\nu}_1(\tilde x ')   \\  & \quad  + C( \veps^4 + W_\infty(\mu, \mu_n))
\\& \leq \delta_0\psi(d_\M(x,y)/\delta_0) - \frac{1}{2} \psi'\left( \frac{d_\M(x,y)}{\delta_0} \right) d_\M(x,y) \int  d_\M(x, \tilde x')^2 K(v, \log_x(\tilde x')) d \mu_x^\M(\tilde x ')   \\  & \quad  + C( \veps^4 + W_\infty(\mu, \mu_n))
\\& \leq \delta_0\psi(d_\M(x,y)/\delta_0) - \psi'\left( \frac{d_\M(x,y)}{\delta_0} \right) d_\M(x,y) \veps^2 \frac{\Ric_x(v)}{2(m+2)}   \\  & \quad  + C( \veps^4 + W_\infty(\mu, \mu_n)).
\end{split}
   \label{eq:Bound_pi_1}
\end{align}
In the second to last inequality we have substituted the integral with respect to $\tilde{\nu}_1 $ with an integral with respect to $\mu_x^\M$ by introducing an error that is of much smaller order than $W_\infty(\mu, \mu_n) + \veps^4$, thanks to Lemma \ref{lem:add1};  in the last inequality we have used \eqref{eqn:DistanceH} and \eqref{eq:AuxTheorem1}.


Next, we find a bound for  $W_{1,G}(  \mu_x^G \lfloor_{B_\M(x,\veps) \setminus B_\M(x,\veps')}, \mu_{y,2}^G)$. We observe that for every $\tilde x \in B_G(x,\veps)$ we have $d_{G,\M}(\tilde x,x) \leq \max \{ \delta_0 , d_\M(\tilde x, x )  \}\leq \veps$. Likewise, for every $\tilde y \in B_G(y,\veps)$ we have $d_{G,\M}(\tilde y,y) \leq \max \{ \delta_0 , d_\M(\tilde y, y )  \}\leq \veps$. Additionally, $d_{G,\M}(x,y) \leq d_\M(x,y) \leq \delta_1  = c_1\veps$. It follows that $d_{G, \M}(\tilde x , \tilde y) \leq C\veps$ for all $\tilde x \in B_G(x,\veps) $ and $\tilde y \in B_G(y,\veps)$.
This implies
\begin{align*}
    W_{1,G}(\mu_x^G \lfloor_{B_\M(x,\veps) \setminus B_\M(x,\veps')}, \mu_{y,2}^G) & \leq C\veps \frac{\mu_n( B_\M(x,\veps) \setminus   B_\M(x,\veps') )}{\mu_n( B_\M(x,\veps) )}
    \\ & \leq C (W_\infty(\mu, \mu_n) + \veps^4), 
\end{align*}
thanks to Lemma \ref{lem:add1}.

We may now invoke Lemma \ref{lem:SplittingMass} and \eqref{eq:Bound_pi_1} to get 
\begin{align}
\begin{split}
   W_{1,G}(\mu_x^G, \mu_y^G ) &\leq W_{1,G}(\mu_x^G \lfloor_{B_\M(x, \veps')}, \mu_{y,1}^G )  + W_{1,G}(\mu_x^G \lfloor_{B_\M(x,\veps) \setminus B_\M(x,\veps')}, \mu_{y,2}^G)  
   \\& \leq \delta_0\psi(d_\M(x,y)/\delta_0) - \psi'\left( \frac{d_\M(x,y)}{\delta_0} \right) d_\M(x,y) \veps^2 \frac{\Ric_x(v)}{2(m+2)}   \\  & \quad  + C( \veps^4 + W_\infty(\mu, \mu_n)).
   \end{split}
\label{eqn:AuxTheorem1WassersteinBoundCase2}
\end{align}
Recalling that $d_{G, \M}(x,y) = \delta_0\psi( d_\M(x,y)/\delta_0) \geq \delta_0 \psi(0) = c_0\psi(0) \veps$, we deduce 
\begin{align}
\begin{split}
\frac{\kappa_G(x,y)}{\veps^2} =  \frac{1}{\veps^2}\left(1 - \frac{W_{1,G}(\mu_x^G, \mu_y^G)}{d_{G,\M}(x,y)} \right) &  \geq   \psi'\left(\frac{d_\M(x,y)}{\delta_0}\right) \frac{d_\M(x,y) }{\delta_0\psi\left( \frac{d_\M(x,y)}{\delta_0 } \right) }\frac{\Ric_x(v)}{2(m+2)}  
\\& - C\left (\veps + \frac{W_\infty(\mu, \mu_n)}{\veps^3}\right) .
\end{split}
    \label{aux:Main1Case2}
\end{align}

\nc

\medskip





Under the assumption that $\Ric_x(v) \geq 2(D+2) K\geq 0$, the first term on the right hand side can be bounded from below by $\psi'(0) \frac{c_0\psi(0) }{12c_1C_\M} K$.
If, on the other hand, $K <0$, then the first term on the right hand side of \eqref{aux:Main1Case2} can be bounded from below by $\frac{c_1}{c_0\psi(0)} K$.

\nc

 


  \textbf{Case 3:} Here we assume that $\delta_1 \geq  d_\M(x,y) \geq \delta_1 -  2 \veps. $ 
  
  According to Assumption \ref{assump} we have $d_\M(x,y) \geq 2 \delta_0$ and in particular $d_{G, \M}(x,y) = \delta_0 \psi\left( \frac{d_\M(x,y)}{\delta_0} \right) = d_\M(x,y)$. Let $\overline{x}$ be the midpoint between $x$ and $y$ along a (manifold) minimizing geodesic connecting them; $\overline{x}$ may not be a point in $\X$, but this is unimportant for our argument. Now, notice that $d_\M(x, \overline{x})= \frac{1}{2} d_\M(x, y) \in [ 2 \delta_0, \delta_1- 2 \veps] $ and also $d_\M(\overline{x},y)= \frac{1}{2} d_\M(x, y) \in  [ 2 \delta_0, \delta_1- 2 \veps]. $ Using the triangle inequality for $W_{1,G}$ and recalling Remark \ref{rem:GeneralizedOllivierBalls} we get:
 \[ W_{1,G}(\mu_x^G, \mu_y^G) \leq W_{1,G}(\mu_x^G, \mu_{\overline x}^G) + W_{1,G}(\mu_{\overline x}^G, \mu_y^G).\]
Then
\begin{align*}
 \kappa_G(x,y) &  = 1 - \frac{W_{1,G}(\mu_x^G, \mu_y^G)}{d_{G,\M}(x,y)}
 \\& \geq  1 - \frac{W_{1,G}(\mu_x^G, \mu_y^G)}{d_{\M}(x,y)}
 \\& \geq 1- \frac{W_{1,G} (\mu_x^G, \mu_{\overline x}^G) + W_{1,G} ( \mu_{\overline x}^G, \mu_y^G)   }{d_\M(x,y)} 
 \\ & = \frac{1}{ 2} \left( 1 - \frac{W_{1,G} (\mu_x^G, \mu_{\overline x}^G)   }{d_\M(x,\bar x)}  \right)   +  \frac{1}{2} \left( 1 - \frac{W_{1,G} (\mu_{\bar x}^G, \mu_{y}^G)   }{d_\M(\bar x,y)}  \right).
\end{align*}
 Using \eqref{eqn:AuxTheorem1WassersteinBoundCase2} twice (which can be applied regardless of whether $\overline{x} \in \X$ or not), and noticing that $\psi(d_\M(x, y) / 2 \delta_0 ) =  d_\M(x, y) / 2 \delta_0  $ and $\psi'( d_\M(x, y) / 2 \delta_0 ) =1$,
 we can lower bound each of the terms on the right hand side of the above expression by $ \frac{1}{2}(s_K\veps^2 K  - C\left (\veps^3 + \frac{W_\infty(\mu, \mu_n)}{\veps}\right) )$.

 \end{proof}
 
\begin{remark}
\label{rem:OnPsi}

We would like to highlight the different ways in which $W_{1,G}(\mu_x^G, \mu_y^G)$ is bounded in Cases 1 and 2 in the previous proof. Indeed, in Case 1, when $d_\M(x,y)$ is very small, we choose a coupling between $\mu_x^G$ and $\mu_y^G$ that leaves most mass fixed, taking advantage of the fact that the overlap between $B_G(x,\veps)$ and $B_G(y, \veps)$ is large in this case. In Case 2, on the other hand, the coupling that we use mimics the coupling in the proof of Theorem \ref{thm:ollivier1}, where all mass moves parallel to the geodesic connecting $x$ and $y$. Notice that we do need to split into these two cases: in going from \eqref{eqn:AuxTheorem1WassersteinBoundCase2} to the final lower bound in Case 2 we need to have a lower bound on $d_\M(x, y)$ that is $O(\veps)$ (for the case $K>0$).

Notice also that the profile function $\psi$ can not be taken to be the identity map for all $t>0$. Indeed, when we divide $W_{1,G}(\mu_x^G, \mu_y^G)$ by $\delta_0 \psi(0)$ to go from \eqref{eqn:AuxTheorem1WassersteinBoundCase2} to \eqref{aux:Main1Case2}, we need $\psi(0) >0$ to guarantee that the term $ \frac{1}{\delta_0 \psi(0 ) \veps^2}(\veps^4+ W_\infty(\mu, \mu_n))$ is indeed small regardless of how small $d_\M(x,y)$ may be. Since the minimum interpoint distance in a data set is much smaller than $O(1/n^{1/m})$, the distance $d_\M(x,y)$ may indeed be quite small. This forces us to consider a profile function $\psi$ that bends away, smoothly (so that the first order Taylor expansion of $\psi$ can reveal the desired curvature term), from the diagonal. The factor $s_K$ in the lower bound \eqref{thm:GlobalBounds1} arises when lower bounding $\kappa_G$ for $x,y$ for which $ O(\delta) \leq   d_\M(x,y) \leq \delta_0$ . We can think of this range as the transition from the Riemannian lengthscale, where $\M$'s geometry can be captured, to a lengthscale where the RGG exhibits complete graph behavior. A somewhat similar separation of scales in an RGG was used in \citep{garcia2020gromov} to study the convergence of discrete Wasserstein spaces defined over RGGs toward the standard Wasserstein space; see the discussions in Remark 1.16. and section 2.1 in \citep{garcia2020gromov}.
\end{remark}
\nc

We now proceed to prove Theorem \ref{thm:GlobalBounds2}. The proof is very similar to the one for Theorem \ref{thm:GlobalBounds1} and thus we will mostly provide details for the steps that need some adjustments. In particular, we highlight the reason for requiring $\hat{d}_g$ to approximate $d_\M$ up to an error of order four; see Remark \ref{rem:WhyBetterDistApprox} below.

\begin{proof}[Proof of Theorem \ref{thm:GlobalBounds2}]

Thanks to Lemma \ref{lem:d_GDistance} we can assume, without the loss of generality, that $x,y \in \X$ are such that $d_{G , \X}(x,y) = \tilde{d}_{G, \X}(x,y) = \delta_0\psi\left( \frac{\hat{d}_g(x,y)}{\delta_0} \right) \leq \delta_1$. As in Theorem \ref{thm:GlobalBounds1} we split our analysis into three different cases. We recall that Ollivier balls in this setting take the form:
\begin{equation}
    B_G(x, \veps) = \{  \tilde x \in \X \: : \: \hat{d}_g(x, \tilde x ) \leq \veps \}, \quad  B_G(y, \veps) = \{  \tilde y \in \X \: : \: \hat{d}_g(y, \tilde y ) \leq \veps \}. 
    \label{eq:ReminderOllivierBalls}
\end{equation}

\textbf{Case 1:} $0<\hat d_g(x,y) \leq  \frac{\psi(0)}{12 C_\M} \delta_0$, where $C_\M$ is as in \eqref{eq:BallDifference}.

We may assume without the loss of generality that $| B_G(x,\veps)| \geq | B_G(y,\veps)| $, for otherwise we can swap the roles of $x$ and $y$. For $\veps_{\pm} := \veps \pm (C_1\beta \veps^3 + C_2 \veps^4)$, thanks to \eqref{eqn:ApproxDistance_HigherOrder}  and  Lemmas \ref{lem:add1} and \ref{lem:OverlapBalls2} we can assume 
\[  \mu_x^G( B_G(x, \veps) \setminus (B_G(x, \veps) \cap B_G(y, \veps) ) ) \leq  \frac{36 \psi(0) c_0}{12^2}. \]
Now, for all $\tilde x \in B_G(x,\veps) $ and $\tilde y \in B_G(y, \veps)$ we have
\[  d_{G, \X}(\tilde x, \tilde y ) \leq d_{G, \X}( x, \tilde x )  + d_{G, \X}( \tilde x, \tilde y )  +  d_{G, \X}( \tilde y, y )  \leq 2 \delta_0 \psi(\veps /\delta_0)  + \delta_0 \leq 3\veps.  \]
By selecting a coupling between $\mu_x^G$ and $\mu_y^G$ that leaves all mass of $\mu_x^G$ in $B_G(x,\veps) \cap B_G(y,\veps)$ fixed, the above estimates imply
\[ W_{1,G}(\mu_x^G, \mu_y^G) \leq   3 \veps  \mu_x^G( B_G(x, \veps) \setminus (B_G(x, \veps) \cap B_G(y, \veps) ) ) \leq  \frac{3}{4}\psi(0) \delta_0 . \]
In addition, since by definition we have $d_{G,\X}(x,y) = \tilde{d}_{G, \X}(x,y) \geq \delta_0 \psi(0)$, it follows
\[ \frac{\kappa_G(x,y)}{\veps^2} = \frac{1}{\veps^2} \left(1- \frac{W_{1,G}(\mu_x^G, \mu_y^G)}{d_{G, \X}(x,y)}\right) \geq \frac{1}{4\veps^2}. \]

\textbf{Case 2:}   $ \frac{\psi(0)}{12 C_\M} \delta_0 \leq \hat{d}_g(x,y) \leq \delta_1 - 2 \veps $.

As in the proof of Theorem \ref{thm:GlobalBounds1} we may further assume, without the loss of generality, that 
\[ a:= \frac{\mu( B_\M(x,\veps))}{ \mu_n( B_G(x,\veps) ) } \leq \frac{\mu( B_\M(y,\veps))}{ \mu_n( B_G(y,\veps) ) }.\]
The measure $\mu_x^G$ is decomposed as
\[ \mu_x^G = \mu_x^G\lfloor_{B_\M(x, \veps')} +  \mu_x^G \lfloor_{B_G(x,\veps) \setminus B_\M(x,\veps')}  ,\]
where now $\veps' := \veps - 3 W_\infty(\mu, \mu_n) - C_1 \beta \veps^3 - (C_2+C_\M')\veps^4$ and, we recall, $B_G$ is as in \eqref{eq:ReminderOllivierBalls}. Notice that the additional terms in the definition of $\veps'$, relative to how $
\veps'$ was defined in the proof of Theorem \ref{thm:GlobalBounds1}, account for the discrepancy between $d_\M$ and $\hat{d}_g$. With this definition we have $B_\M(x, \veps' ) \cap \X \subseteq B_G(x,\veps)$.

We define the measure $\mu_{y,1}^G$ as
 \[ \mu_{y,1}^G :=  a T_{n \sharp} ( \F_{\sharp }  (  \mu_x^\M \lfloor_{T_n^{-1}(B_\M(x,\veps'))} ) ),     \]
for $T_n, \F$ and $\mu_x^\M$ as in Case 2 in the proof of Theorem \ref{thm:GlobalBounds1}. We can follow the same steps there to conclude that $\mu_{y,1}^G \leq \mu_y^G$ and then define $\mu_{y,2}^G:= \mu_y^G - \mu_{y,1}^G$. Also, we may introduce analogous couplings $\Pi$ and $\pi_1^G = T_{1,4\sharp}  \Pi \in \Gamma( \mu_x^G\lfloor_{B_\M(x, \veps')} , \mu_{y,1}^G)$ for which:
\[ d_\M(\tilde x, \tilde y) = d_\M(x,y) \left( 1 - d_\M(x, \tilde x ')^2\left(\frac{K(v,w')}{2}  + O( d_\M(x,y) + \veps) \right) \right) +  O(W_\infty(\mu, \mu_n))   \]
for all $(\tilde x, \tilde x' ,\tilde y ', \tilde y ) \in \supp(\Pi)$. In turn, we can use the approximation error estimates for  $\hat{d}_g-d_\M$ to obtain
\[ \hat{d}_g(\tilde x, \tilde y) = \hat{d}_g(x,y) \left( 1 -  d_\M(x, \tilde x ')^2 \left(\frac{K(v,w')}{2}  + O( \hat{d}_g(x,y) + \veps) \right) \right) + O( \beta \veps^3 + \veps^4 + W_\infty(\mu, \mu_n))   \]
for all $(\tilde x, \tilde x' ,\tilde y ', \tilde y ) \in \supp(\Pi)$, and $d_{G, \X}(\tilde x , \tilde y) \leq \tilde{d}_{G, \X}(\tilde x , \tilde y) = \delta_0 \psi( \hat{d}_g(x,y)/\delta_0)$ for all $(\tilde x , \tilde y) \in \supp(\pi_1^G)$. From this we can conclude that
\begin{align}
\begin{split}
  & W_{1,G}(\mu_x^G \lfloor_{B_\M(x, \veps')}, \mu_{y,1}^G )  \leq \int d_{G,\X}(\tilde x , \tilde y) d \pi_1^G(\tilde x , \tilde y)
 =  \int d_{G,\X}(\tilde x, \tilde y) d \Pi(\tilde x ,\tilde x', \tilde y', \tilde y)
\\& \leq \delta_0\psi(\hat{d}_g(x,y)/\delta_0) - \psi'\left( \frac{\hat{d}_g(x,y)}{\delta_0} \right) \hat{d}_g(x,y) \veps^2 \frac{\Ric_x(v)}{2(m+2)}   \\  & \quad  + C( \beta \veps^3 + \veps^4 + W_\infty(\mu, \mu_n)).
\end{split}
   \label{eq:Bound_pi_1DataDriven}
\end{align}
In addition, 
\begin{align*}
    W_{1,G}(\mu_x^G \lfloor_{B_\M(x,\veps) \setminus B_\M(x,\veps')}, \mu_{y,2}^G) & \leq C\veps \frac{\mu_n( B_\M(x,\veps) \setminus   B_\M(x,\veps') )}{\mu_n( B_\M(x,\veps) )}
    \\ & \leq C (W_\infty(\mu, \mu_n) + \beta \veps^3 + \veps^4), 
\end{align*}
thanks to Lemma \ref{lem:add1}.

We may now invoke Lemma \ref{lem:SplittingMass} and \eqref{eq:Bound_pi_1DataDriven} to get 
\begin{align}
\begin{split}
   W_{1,G}(\mu_x^G, \mu_y^G ) &\leq W_{1,G}(\mu_x^G \lfloor_{B_\M(x, \veps')}, \mu_{y,1}^G )  + W_{1,G}(\mu_x^G \lfloor_{B_\M(x,\veps) \setminus B_\M(x,\veps')}, \mu_{y,2}^G)  
   \\& \leq \delta_0\psi(\hat{d}_g(x,y)/\delta_0) - \psi'\left( \frac{\hat{d}_g(x,y)}{\delta_0} \right) \hat d_g(x,y) \veps^2 \frac{\Ric_x(v)}{2(m+2)}   \\  & \quad  + C( \beta \veps^3 +\veps^4 + W_\infty(\mu, \mu_n)).
   \end{split}
\label{eqn:AuxTheorem1WassersteinBoundCase2DataDriven}
\end{align}
Recalling that $d_{G, \X}(x,y) = \delta_0\psi( \hat d_g(x,y)/\delta_0) \geq \delta_0 \psi(0) = c_0\psi(0) \veps$, we deduce 
\begin{align}
\begin{split}
\frac{\kappa_G(x,y)}{\veps^2} =  \frac{1}{\veps^2}\left(1 - \frac{W_{1,G}(\mu_x^G, \mu_y^G)}{d_{G,\X}(x,y)} \right) &  \geq   \psi'\left(\frac{\hat d_g(x,y)}{\delta_0}\right) \frac{\hat d_g(x,y) }{\delta_0\psi\left( \frac{\hat d_g(x,y)}{\delta_0 } \right) }\frac{\Ric_x(v)}{2(m+2)}  
\\& - C\left (\beta+ \veps + \frac{W_\infty(\mu, \mu_n)}{\veps^3}\right) .
\end{split}
\label{aux:Main1Case2DataDriven}
\end{align}
The lower bound \eqref{eqn:LowerboundTheorem2} now follows.

\nc

\textbf{Case 3:} Here we assume that $ \delta_1 -  2 \veps \leq   \hat{d}_g(x,y) \leq \delta_1$.

As in Case 3 in the proof of Theorem \ref{thm:GlobalBounds1} we consider the midpoint $\overline{x}$ between $x$ and $y$ (along the manifold geodesic). It is straightforward to see from 1 in Assumption \ref{assump:hatdg} that
\begin{equation}
   \left | \frac{\hat{d}_g(\overline{x}, y)}{\hat d_g(x,y)}  - \frac{1}{2} \right| \leq C \beta \veps^2 + C \veps^3, \quad \left | \frac{\hat{d}_g(\overline{x}, x)}{\hat d_g(x,y)}  - \frac{1}{2} \right| \leq C \beta \veps^2 + C \veps^3.  
   \label{aux:Main2Case3}
\end{equation}
Then
\begin{align*}
 \kappa_G(x,y) & \geq 1- \frac{W_{1,G} (\mu_x^G, \mu_{\overline x}^G) + W_{1,G} ( \mu_{\overline x}^G, \mu_y^G)   }{d_{G, \X}(x,y)} 
 \\ & = 1- \frac{W_{1,G} (\mu_x^G, \mu_{\overline x}^G) + W_{1,G} ( \mu_{\overline x}^G, \mu_y^G)   }{\hat{d}_g(x,y)} 
 \\ & \geq \frac{\hat{d}_g(x, \bar x)}{ \hat d_g(x,y) } \left( 1 - \frac{W_{1,G} (\mu_x^G, \mu_{\overline x}^G)   }{\hat d_g(x,\bar x)}  \right)   +  \frac{\hat d_g(\bar x, y)}{ \hat d_g(x,y) } \left( 1 - \frac{W_{1,G} (\mu_{\bar x}^G, \mu_{y}^G)   }{\hat d_g(\bar x,y)}  \right)
 \\ & - C \beta \veps^2 - C\veps^3.
\end{align*}
As in the proof of Theorem \ref{thm:GlobalBounds1}, we may now use \eqref{aux:Main1Case2DataDriven} to bound from below each of the terms $\left( 1 - \frac{W_{1,G} (\mu_x^G, \mu_{\overline x}^G)   }{\hat d_g(x,\bar x)}  \right) $ and $\left( 1 - \frac{W_{1,G} (\mu_y^G, \mu_{\overline x}^G)   }{\hat d_g(y,\bar x)}  \right) $ by $ \frac{1}{2}(s_K\veps^2 K  - C\left ( \beta \veps^2 + \veps^3 + \frac{W_\infty(\mu, \mu_n)}{\veps}\right) )$.



\end{proof}

\begin{remark}
\label{rem:WhyBetterDistApprox}
In the regime $\hat{d}_g(x,y) \sim   \veps$ , i.e., the regime corresponding to Case 2 in the proof of Theorem \ref{thm:GlobalBounds2}, we use the fact that $\hat{d}_g$ satisfies
\[ \hat{d}_g(x,y) = d_\M(x,y) + O(\beta \veps^3 + \veps^4), \]
whereas an approximation error of order $O(\veps^3)$ would have produced a lower bound on discrete curvature of the form $s_K K - C $ for some constant $C$ that may be larger than $s_K K$ itself. In particular, if the error was of order $O(\veps^3)$, the sign of the discrete lower bound would not be guaranteed to be consistent with the sign of the manifold's curvature lower bound. From our proof it thus seems that $\hat{d}_g(x,y)$ cannot be simply taken to be the Euclidean distance between $x$ and $y$ and a finer estimator seems to be necessary.
\end{remark}




  \subsection{Pointwise consistency}
  \label{sec:LocalConsistency}

  Next, we present the proof of our pointwise consistency results, i.e., Theorems \ref{thm:consistency-non-asymp} and \ref{thm:consistency-non-asymp-geo}. 

\begin{proof}[Proof of Theorem \ref{thm:consistency-non-asymp}]

Since $d_\M(x,y)$ is assumed to satisfy $2\delta_0 \leq d_\M(x,y) \leq \frac{1}{2}\delta_1 $, we may use Proposition \ref{prop:ComparisonMetrics},  \eqref{aux:Main1Case2} in Case 2 in the proof of Theorem \ref{thm:GlobalBounds1}, and the fact that $\psi(t)=t$ for $t \geq 1$ to conclude that 
\begin{align*}
\frac{\kappa_G(x,y)}{\veps^2} \geq   \frac{\Ric_x(v)}{2(m+2)}  - C\left (\veps + \frac{W_\infty(\mu, \mu_n)}{\veps^3}\right) .
\end{align*}
It thus remains to obtain matching upper bounds. 

For this purpose, let $f: \M \rightarrow \R  $ be the function defined in \eqref{eqn:functionF}. Using \eqref{eq:Comparisson1} in Proposition \ref{prop:ComparisonMetrics} and the fact that $f$ is $1$-Lipschitz with respect to $d_\M$, we conclude that the function $f$ restricted to $\X$ is $1$-Lipschitz with respect to $d_{G,\M}$. In turn, thanks to the Kantorovich-Rubinstein theorem (i.e., Theorem \ref{thm:KantorovichRubinstein}) we obtain
\[  \int f(\tilde y) d\mu_y^G(\tilde y) -  \int f(\tilde x) d\mu_x^G(\tilde x)  \leq W_{1,G}(\mu_x^G, \mu_y^\G).  \]
Using again the fact that $f$ is $1$-Lipschitz with respect to $d_\M$ we deduce
\[  \left| \int f(\tilde x) d\mu_x^G(\tilde x) -  \int f(\tilde x) d\mu_x^\M(\tilde x) \right| \leq W_1(\mu_x^G, \mu_x^\M), \]
and 
\[  \left| \int f(\tilde x) d\mu_y^G(\tilde y) -  \int f(\tilde y) d\mu_y^\M(\tilde y) \right| \leq W_1(\mu_y^G, \mu_y^\M). \]
Putting together the above inequalities we conclude that 
\[  \int f(\tilde y) d\mu_y^\M(\tilde y) - \int f(\tilde x) d\mu_x^\M(\tilde x)   \leq   W_{1,G}(\mu_x^G, \mu_y^G)   +  W_1(\mu_x^G, \mu_x^\M) + W_1(\mu_y^G, \mu_y^\M).   \]
Using now \eqref{eq:FormulafRic}, we can lower bound the left hand side of the above expression and conclude that 
\begin{equation}
\label{eq:AuxUpperBounds1}
   1 -\veps^2  \frac{\Ric_x(v)}{2(m+2)}  \leq \frac{W_{1,G}(\mu_x^G, \mu_y^G)}{d_\M(x,y)} + \varphi   , 
\end{equation}
where 
\[  \varphi :=  C (d_\M(x,y) \veps^2 +  \veps^3 )  +  \frac{W_1(\mu_x^G, \mu_x^\M)}{d_\M(x,y)} + \frac{W_1(\mu_y^G, \mu_y^\M)}{d_\M(x,y)}.  \]
Using the fact that $d_{G,\M}(x,y) = d_\M(x,y)$ by Proposition \ref{prop:ComparisonMetrics}, and rearranging terms, we conclude that
\[ \kappa_G(x,y) \leq \veps^2 \frac{\Ric_x(v)}{2(m+2)} + C (d_\M(x,y) \veps^2 + \veps^3  +  \frac{W_1(\mu_x^G, \mu_x^\M)}{d_\M(x,y)} + \frac{W_1(\mu_y^G, \mu_y^\M)}{d_\M(x,y)} ).    \]
To finish the proof, it remains to notice that the terms $W_1(\mu_x^G, \mu_x^\M)$ and $W_1(\mu_y^G, \mu_y^\M)$ can be bounded above by $C W_\infty(\mu, \mu_n) $, as it follows easily from an application of Lemma \ref{lem:add1}.

\end{proof}


\begin{proof}[Proof of Theorem \ref{thm:consistency-non-asymp-geo}]

From Case 2 in the proof of Theorem \ref{thm:GlobalBounds2} we have
\[ 
\frac{\kappa_G(x,y)}{\veps^2} \geq   \frac{\Ric_x(v)}{2(m+2)}  - C\left (\beta+ \veps + \frac{W_\infty(\mu, \mu_n)}{\veps^3}\right) ,\]
and thus it remains to obtain a matching upper bound.

First of all, notice that, thanks to \eqref{eqn:ApproxDistance_HigherOrder} and Assumption \ref{assump}, we can assume that $ 2 \delta_0 \leq d_\M(x,y) \leq \frac{\delta_1}{2}$. Now, from \eqref{eq:Comparisson2} and \eqref{eq:Comparisson3} in Proposition \ref{prop:ComparisonMetrics}  we have
\begin{equation}
\label{eq:AuxUpperBounds2}
   \left| \frac{d_\M(x,y)}{d_{G,\X}(x,y)}  -1   \right| \leq C(\beta \veps^2+\veps^3).
\end{equation}
On the other hand, thanks to \eqref{eq:Comparisson2}, it follows that the function $f$ from \eqref{eqn:functionF} restricted to $\X$ has Lipschitz constant, with respect to $d_{G, \X}$, no larger than $1+ C(\beta \veps^2 + \veps^3)$. This implies that 
\[ \int f(\tilde y) d\mu_y^G(\tilde y) -  \int f(\tilde x) d\mu_x^G(\tilde x)  \leq  (1+ C(\beta \veps^2 + \veps^3) )W_{1,G}(\mu_x^G, \mu_y^\G).\]
Proceeding as in the proof of Theorem \ref{thm:consistency-non-asymp} we can conclude that
\begin{equation*}
   1 -\veps^2  \frac{\Ric_x(v)}{2(m+2)}  \leq ( 1+ C(\beta \veps^2 +  \veps^3))\frac{W_{1,G}(\mu_x^G, \mu_y^G)}{d_\M(x,y)} + \varphi   , 
\end{equation*}
for $\varphi$ as in \eqref{eq:AuxUpperBounds1}. In turn, multiplying both sides of the above by $ \frac{1}{1+ C(\beta \veps^2 + \veps^3)} \frac{d_\M(x,y)}{d_{G, \X}(x,y)} $, using \eqref{eq:AuxUpperBounds2}, and rearranging terms, we conclude that
\[ \kappa_G(x,y) \leq \veps^2 \frac{\Ric_x(v)}{2(m+2)} + C (\beta \veps^2 + \veps^3  +  \frac{W_1(\mu_x^G, \mu_x^\M)}{d_\M(x,y)} + \frac{W_1(\mu_y^G, \mu_y^\M)}{d_\M(x,y)} ).    \]
The result now follows from the fact that, thanks to \eqref{eqn:ApproxDistance_HigherOrder} and Lemma \eqref{lem:add1},  each of the terms $W_1(\mu_x^G, \mu_x^\M)$ and $W_1(\mu_y^G, \mu_y^\M)$ is bounded by $C(W_\infty(\mu, \mu_n) + \beta \veps^3 + \veps^4)$.

\end{proof}

    \nc

   
    

\section{Applications}
\label{sec:Applications}

\subsection{Lipschitz contractivity of the graph heat kernel }
\label{sec:LaplaciansandRegression}

In this section we discuss some of the implications of our curvature lower bounds on the heat kernel associated to the \textit{unnormalized graph Laplacian} $\Delta_n$ induced by the graph $G=(\X, w_{\veps})$. We recall that the unnormalized graph Laplacian associated to $G$ is defined as
\begin{equation}
\label{def:UnnormalizedGraphLaplacian}
\Delta_n u(x) := \frac{1}{n\veps^{m+2}} \sum_{\tilde x \in \X} \omega_\veps(x,\tilde x ) ( u(x) - u(\tilde x)), \quad u\in L^2(\X).
\end{equation}
We will focus on the choice $w_\veps= w_{\veps, \M}$ (see \eqref{eqn:WeightsManifold}) for simplicity, but we remark that a lot of the discussion presented below can be adapted to the choice $w_\veps = w_{\veps, \X }$.

\begin{remark}
The operator $\Delta_n$ can be written in matrix form as
\[ \Delta_n = \frac{1}{\veps^{2}} (D - W), \]
where $W$ is the weight matrix induced by the rescaled weights $\frac{1}{n\veps^m }w_\veps$ and $D$ is the degree matrix associated to $W$. 
\end{remark}


$\Delta_n$ plays a central role in graph-based learning, where it is used to define algorithms for supervised, semi-supervised, and unsupervised learning; see, e.g., \citep{belkin2003laplacian,Nadler2006_Diffusion,VonLuxburg2007} for some discussion. There are several results in the literature that discuss the asymptotic convergence of $\Delta_n$ toward $\M$'s Laplace-Beltrami operator; see e.g. \citep{hein_graphs_2005} for pointwise convergence and \citep{burago2015graph,calder_improved_2020,CalderNGTLewicka,DUNSON2021282,Trillos2019_Error,WormelReich} for spectral convergence. Here we add upon the existing literature on graph Laplacians by providing novel contraction results that are implied by our curvature lower bounds. Specifically, we are interested in the behavior of the heat operator $e^{-t \Delta_n}$ as $t \rightarrow \infty$. The heat operator $e^{-t \Delta_n}$ can be defined via the spectral theorem or as the operator mapping an initial condition $u \in L^2(\X)$ to the solution at time $t$ of the graph heat equation:
\begin{align}
 \begin{cases}\partial_s u_s = -\Delta_n u_s,
 \\ u_0=u.
 \end{cases}
\end{align}

In what follows we abuse notation slightly and use $D(x)$ to denote the degree of $x\in \X$. Precisely,
\[ D(x) := \frac{1}{n \veps^m} \sum_{y \in \X} \eta\left( \frac{d_\M(x,y)}{\veps} \right),  \]
where $\eta(t):= \mathds{1}_{ t \leq 1}$. $D$ can be thought of as a kernel density estimator for the distribution used to sample the data set $\X$, in this case the uniform measure over $\M$. Indeed, one can show, via standard concentration arguments, that for every $r \in [\veps^2 , 1]$ we have
\begin{equation}
\label{eqn:DegreeApproximation}
   \max_{x \in \X} \left | \alpha_\M - D(x) \right| \leq Cr,  
\end{equation}
 with probability at least $1 - c(r\veps)^{-m } \exp(-c r^2 n \veps^m)$; e.g., see Corollary 3.7 in \citep{CalderNGTLewicka} for a closely related estimate. In the above, $\alpha_\M \vol(\M)$ is the volume of the $m$-dimensional Euclidean unit ball.

We study the evolution, along the heat flow, of the Lipschitz seminorm of a function $u: \X\rightarrow \R $ when $\X$ is endowed with the distance $d_{G}= d_{G, \M}$. This seminorm is defined as:
\begin{equation}
   \mathrm{Lip}_G(u):= \max_{x, y \in \X ,\:  x \not = y}  \frac{ | u(x) - u(y)| }{d_{G}(x,y) }.
\end{equation}

\begin{lemma}
	\label{cor:ContractionLipschitz}
For a given $u : \X \rightarrow \R  $ let
\[ \mathcal{A}u (x)  :  = \int u(\tilde x) d\mu_x^G(\tilde x ) = \frac{1}{ n\veps^m D(x)} \sum_{z \in \X} w_{\veps}(x,z) u(z) , \quad x \in \X .\]
Under the same assumptions as in Theorem \ref{thm:GlobalBounds1} it follows that
\[ \mathrm{Lip}_G(  \mathcal{A}u ) \leq  (1 - \veps^2 K_G)  \mathrm{Lip}_G(u) , \quad \forall u \in L^2(\X),  \]
where
\[ K_G := \min \left\{ s_K K  - C\left (\veps + \frac{W_\infty(\mu, \mu_n)}{\veps^3}\right), \frac{1}{2\veps^2} \right\}. \]
\end{lemma}
\begin{proof}
This is an immediate consequence of the definition of Ollivier Ricci curvature and the dual representation of the $1$-Wasserstein distance. Indeed, by Theorem \ref{thm:GlobalBounds1} and the Kantorovich-Rubinstein theorem, for all $x, y \in \X$ we have
\begin{align*}  
(1- \veps^2 K_G) d_G(x,y)  &  \geq  W_1(\mu_x^G , \mu_y ^G  )   \geq \frac{1}{\text{Lip}_G(u)} \left(  \int u( \tilde x ) d\mu_x^G(\tilde x )  -  \int u( \tilde y ) d\mu_y^G(\tilde y )    \right)
\\& =  \frac{1}{\text{Lip}_G(u)} ( \mathcal{A}u (x) - \mathcal{A}u(y)).
\end{align*}
Since the above is true for all $x, y \in \X $ we obtain the desired result.
\end{proof}

Using Lemma \ref{cor:ContractionLipschitz} we can establish the following contraction of $\text{Lip}_G$ along the heat flow $e^{-t\Delta_n}$.

\begin{theorem}
\label{thm:LipschitzContraction}
Under the same assumptions as in Theorem \ref{thm:GlobalBounds1}, and letting $K_G$ be defined as in Lemma \ref{cor:ContractionLipschitz}, for all $u: \X \rightarrow \R$ we have 
\begin{equation}
  \mathrm{Lip}_G (e^{-t \Delta_n} u  ) \leq \exp\left( -\left(K_G - \frac{4\lVert D - \alpha_\M \rVert_{L^\infty(\X)} \mathrm{diam}(G) } {c_0 \psi(0)\veps^3}\right) t \right) \mathrm{Lip}_G(u),
 \label{eqn:LipschitzContraction}
\end{equation}
where $\mathrm{diam}(G):= \max_{x, y \in \X} d_G(x,y)$.
\end{theorem}
\begin{proof}
We start by noticing that inequality \eqref{eqn:LipschitzContraction} is invariant under addition of constants. This is because $e^{-t \Delta_n } (u + c) = e^{-t \Delta_n} u + c  $. Due to this, from now on we can assume without the loss of generality that $u$ is such that $\sum_{z\in \X} u(z)=0$.

Now, fix $t \in [0,\infty)$ and let $x,y \in \X$ be a pair of data points such that
\[  \frac{e^{-t \Delta_n} u (x) - e^{-t\Delta_n}u(y) }{d_G(x,y)} = \text{Lip}_G(e^{-t \Delta_n} u );  \]
such pair always exists because $\X$ is a finite set. Notice that
\begin{equation} 
\label{eqn:CotracAux1}
\frac{d}{dt} \frac{( e^{-t \Delta_n} u(x) - e^{-t \Delta_n} u(y) )^2}{2 d_G(x,y)^2} = \frac{ e^{-t \Delta_n} u(x) - e^{-t \Delta_n} u(y) }{d_G(x,y)^2} \cdot ( - \Delta_n e^{-t \Delta_n} u(x) + \Delta_n e^{-t \Delta_n } u(y)  ). 
\end{equation}
We rewrite the term $\Delta_n e^{-t \Delta_n} u(x)$ as
\begin{align*}
\frac{1}{\veps^2}D(x) e^{-t \Delta_n} u(x) - \frac{1}{n\veps^{m+2}}\sum_{z \in \X} w_\veps(x,z) e^{-t \Delta_n} u(z) &=  \frac{\alpha_\M}{\veps^2}e^{-t \Delta_n  } u (x) -\frac{\alpha_\M}{\veps^2} \mathcal{A} e^{- t \Delta_n}u(x)
\\&+ \frac{1}{\veps^2}(D(x) - \alpha_\M)  e^{-t \Delta_n} u(x) 
\\&+  \frac{1}{n\veps^{m+2}}\left( \frac{\alpha_\M - D(x)}{D(x)}  \right) \sum_{ z \in \X}w_\veps(x,z)  e^{-t \Delta_n} u(z).
 \end{align*}
We plug this expression (and the one corresponding to $\Delta_n e^{-t \Delta_n} u(y)$) in \eqref{eqn:CotracAux1} to conclude that
\begin{align*} 
\frac{d}{dt} \frac{( e^{-t \Delta_n} u(x) - e^{-t \Delta_n} u(y) )^2}{2 d_G(x,y)^2} & \leq - \alpha_\M \frac{(e^{-t \Delta_n} u(x) - e^{-t \Delta_n} u(y))^2  }{\veps^2 d_G(x,y)^2} + \frac{\alpha_\M}{\veps^2} \text{Lip}_G( e^{-t \Delta_n } u ) \text{Lip}_G(\mathcal{A} e^{-t \Delta_n }  u)  
\\& +  \frac{4}{\veps^2 \delta_0 \psi(0)} \text{Lip}_G(e^{-t \Delta_n } u) \lVert D -\alpha_\M \rVert_{L^\infty(\X)} \lVert e^{-t \Delta_n} u \rVert_{L^\infty(\X)} 
\\& \leq -\frac{ \alpha_\M (\text{Lip}_G(e^{-t \Delta_n} u ))^2 }{\veps^2}+ \frac{\alpha_\M}{\veps^2}( 1- \veps^2 K_G) (\text{Lip}_G(e^{-t \Delta_n} u ))^2
\\& + \frac{4}{\veps^2 \delta_0 \psi(0)} \text{Lip}_G(e^{-t \Delta_n } u) \lVert D -\alpha_\M \rVert_{L^\infty(\X)} \lVert e^{-t \Delta_n} u \rVert_{L^\infty(\X)},
\end{align*}
where in the second inequality we have used Corollary \ref{cor:ContractionLipschitz}. By assumption we have $\frac{1}{n } \sum_{z \in\X} e^{-t \Delta_n} u(z) = \frac{1}{n } \sum_{z \in \X}  u(z) =0 $ and thus it follows
\[ | e^{-t \Delta_n}u (z')| = \left|\frac{1}{n}\sum_{z \in \X } ( e^{-t \Delta_n}u (z') - e^{-t \Delta_n}u (z) )\right| \leq \text{diam}(G) \cdot \text{Lip}_G(e^{-t \Delta_n} u )   \]
for every $z'\in \X$. This allows us to bound $\lVert e^{-t \Delta_n} u \rVert_{L^\infty(\X)} \leq \text{diam}(G) \cdot \text{Lip}_G(e^{-t \Delta_n} u )$. Hence
\begin{align*} 
\frac{d}{dt} \frac{( e^{-t \Delta_n} u(x) - e^{-t \Delta_n} u(y) )^2}{ d_G(x,y)^2} & \leq - 2\left( \alpha_\M K_G  - \frac{4\lVert D - \alpha_\M \rVert_{L^\infty(\X)} \text{diam}(G)} {\veps^2 \delta_0 \psi(0)}\right) (\text{Lip}_G(e^{-t \Delta_n} u ))^2.
\end{align*}
Since $(x,y)$ was an arbitrary pair realizing $\text{Lip}_G(e^{-t \Delta_n} u)$ we conclude that
\[ \frac{d}{dt} (\text{Lip}_G(e^{-t \Delta_n} u  ))^2 \leq - 2\left( \alpha_\M K_G - \frac{4\lVert D -\alpha_\M \rVert_{L^\infty(\X)} \text{diam}(G) } {\veps^2 \delta_0  \psi(0)}\right) (\text{Lip}_G(e^{-t \Delta_n} u ))^2.    \]
Gronwall's inequality implies that 
\[  (\text{Lip}_G(e^{-t \Delta_n} u  ))^2 \leq \exp\left( -2\left(\alpha_\M{\kappa}_G - \frac{4\lVert D -\alpha_\M \rVert_{L^\infty(\X)} \text{diam}(G) } {\veps^2 \delta_0 \psi(0)}\right) t \right) (\text{Lip}_G(u  ))^2.  \]
Taking square roots on both sides we obtain the desired result. 
\end{proof}

\begin{remark}
In order for the exponent on the right hand side of \eqref{eqn:LipschitzContraction} to be negative, we certainly need $K_G $ to be strictly positive, which we can guarantee when $\M$ is a manifold with Ricci curvature bounded from below by a positive quantity and the assumptions of Theorem \ref{thm:GlobalBounds1} are satisfied. We also need to make sure that the quantity  $\frac{\lVert D -\alpha_\M \rVert_{L^\infty(\X)} } {\veps^3}$ is sufficiently small, which is implied by the assumptions in Theorem \ref{thm:GlobalBounds1} and the bound \eqref{eqn:DegreeApproximation}. The bottom line is that, when $\X$ is sampled from a manifold with positive Ricci curvature, then, under the assumptions in Theorem \ref{thm:GlobalBounds1}, for all large enough $n$ the Lipschitz seminorm $\mathrm{Lip}_G$ contracts along the heat flow associated to the unnormalized Laplacian for the graph $G=(\X , w_{\veps, \M})$.    
\end{remark}

\begin{remark}

We emphasize that $\Delta_n$ in Theorem \ref{thm:LipschitzContraction} is the unnormalized Laplacian of $G=(\X, w_{\veps,\M})$, which we recall depends on the geodesic distance over $\M$. While our curvature lower bound results do not allow us to say anything about Laplacians for RGGs with the Euclidean distance, one can certainly deduce adaptations of Theorem \ref{thm:LipschitzContraction} for proximity graphs built from slight modifications of the Euclidean distance. In particular, it is clear that a similar statement can be derived, mutatis mutandis, for the graph $G=(\X, w_{\veps, \X})$ endowed with distance $d_{G, \X}$.

\end{remark}

\begin{remark}
\label{rem:SpectralGaps} 
To contrast the content of Theorem \ref{thm:LipschitzContraction} with other contractivity results known in the literature, 
let $\lambda_{G}$ be the first nontrivial eigenvalue of $\Delta_n$. Using the spectral theorem one can easily show that for all $u \in L^2(\X)$
 \[ \lVert e^{-t \Delta_n} u - \overline{u}    \rVert_{L^2(\X)}^2 \leq  e^{- t \lambda_G} \lVert u - \overline{u} \rVert_{L^2(\X)}^2,  \]
 where $\overline{u} = \frac{1}{n}\sum_{z \in \X} u(z)$. Spectral consistency results for $\Delta_n$ like the ones in \citep{burago2015graph,calder_improved_2020,CalderNGTLewicka,DUNSON2021282,Trillos2019_Error,WormelReich} guarantee that $\lambda_G$ does not deteriorate as the graph $G$ is scaled up. Naturally, from these $L^2$ contraction estimates one can not derive Lipschitz contraction as in Theorem \ref{thm:LipschitzContraction} and our results in this paper thus provide novel results for the literature of graph Laplacians on data clouds. It is worth highlighting, however, that for $\lambda_G>0$ to remain bounded away from zero, one does not require $\M$'s Ricci curvature to be positive. 
\end{remark}

\begin{remark}
In the literature on graph based learning it is not unusual to replace a graph Laplacian with a version of it that is obtained by truncating its spectral decomposition, which in particular requires the use of an eigensolver. We emphasize that Theorem \ref{thm:LipschitzContraction} and its Corollary \ref{cor:LinftyDecay} below is a structural property that holds for the full Laplacian $\Delta_n$ and not for a truncation thereof.
\end{remark}

An immediate consequence of Theorem \ref{thm:LipschitzContraction} is the following.
\begin{corollary}
\label{cor:LinftyDecay}
Under the same assumptions as in Theorem \ref{thm:LipschitzContraction} we have
\begin{equation*}
\lVert e^{-t \Delta_n} u - \overline{u}  \rVert_{L^\infty(\X)}\leq  \exp\left( -\left(K_G - \frac{4\lVert D - \alpha_\M \rVert_{L^\infty(\X)} \mathrm{diam}(G) } {c_0 \psi(0)\veps^3}\right) t \right)  \mathrm{diam}(G) \mathrm{Lip}_G(u),
\end{equation*}
where $\overline{u}:=  \frac{1}{n}\sum_{x \in \X} u(x)$.
\end{corollary}

\begin{proof}
 Notice that for any function $v : \X \rightarrow \R$ and any $x \in \X$ we have
 \[ |v(x) - \overline{v}| =\left | \frac{1}{n}\sum_{\tilde x \in \X} ( v(x) - v(\tilde x))  \right| \leq \text{diam}(G) \text{Lip}(v) , \]
from where it follows that
\[  \lVert v - \overline{v} \rVert_{L^\infty(\X)} \leq \text{diam}(G) \text{Lip}_G(v). \]
 The result now follows from Theorem \ref{thm:LipschitzContraction}.
\end{proof}

\nc

\subsection{Manifold Learning}
\label{sec:ManifoldLearning}

We briefly comment on another class of estimation problems on point clouds and graphs where curvature lower bounds may be utilized.

Recognizing and characterizing geometric structure in data is a cornerstone of Representation Learning. A common assumption is that the data lies on or near a low-dimensional manifold $\Mc \subseteq \R^d$ (\emph{manifold hypothesis}). Suppose we are given a point cloud $\Xc \subseteq \R^d$ in a high-dimensional Euclidean space, i.e., our data was sampled from an embedded manifold and we have access to pairwise Euclidean distances between data points. What can we learn about the dimension and curvature of $\Mc$ given only pairwise Euclidean distances in $\Xc$? A rich body of literature has considered this question from different angles. Several algorithms exist for inferring the \emph{intrinsic dimension} of $\Xc$ (i.e., $\dim(\Mc)$). However, such algorithms do not allow inference on intrinsic geometric quantities of $\M$ such as a global curvature bound. There are several approaches for approximating \emph{extrinsic} curvature, some of which are reviewed in Appendix \ref{sec:EstimatesSecondFundForm}. However, none of these techniques allow for learning the \emph{intrinsic} curvature of the manifold directly, i.e., without characterizing the full Riemannian curvature tensor.
The consistency results developed in this paper allow for such inference, even in the case where one has only access to data-driven estimates of pairwise geodesic distances, as is usually the case in practise.

\emph{Manifold learning} aims to identify a putative manifold $\widetilde{\Mc} \subseteq \R^d$ whose geometry agrees with the low-dimensional structure in $\Xc$. That is, one learns a point configuration $\phi(\Xc)$, which is the output of an implicit map $\phi: \Xc \rightarrow \widetilde{\Mc}$, that approximately preserves the pairwise distances ($d_\Xc(x,y) \approx d_{\widetilde{\Mc}}(\phi(x),\phi(y))$ for all $x,y \in \Xc$). A large number of algorithms has been proposed for this task, including Isomap~\citep{isomap}, Laplacian Eigenmaps~\citep{belkin2003laplacian} and Locally Linear Embeddings~\citep{lle}. While these algorithms have gained popularity in practice, it is often challenging to certify that the geometry of the putative manifold $\widetilde{\Mc}$ aligns with that of the true manifold $\Mc$. To the best of our knowledge, the strongest guarantees are available for Isomap, which is known to recover the intrinsic dimension, as well as, asymptotically, the pairwise distances, in the large-sample limit~\citep{bernstein2000graph}. However, none of these manifold learning algorithms are guaranteed to recover global curvature bounds. The consistency of global curvature bounds (Theorems~\ref{thm:GlobalBounds2}) provides an effective, unsupervised means for testing whether $\Mc$ has a curvature lower bound by computing Ollivier's Ricci curvature on a geometric graph constructed from $\Xc$. The resulting tool, complementary to standard manifold learning techniques, could allow for learning a more comprehensive geometric characterization of a given point cloud. Curvature lower bounds may also serve as inductive biases in manifold learning approaches. The choice of manifold learning technique often requires prior knowledge on the type of manifold that is to be learnt, e.g., if the data was sampled from a linear subspace, a linear method, such as Principal Component Analysis, is suitable. On the other hand, if the data is sampled from a nonlinear subspace, such as an embedded submanifold, a nonlinear approach, such as Isomap, is expected to perform better.

\section{Conclusions}
\label{sec:Conclusions}

In this paper, we have investigated continuum limits of Ollivier's Ricci curvature on random geometric graphs in the sense of local pointwise consistency and in the sense of global lower bounds. Specifically, we consider a data cloud $\Xc$ sampled uniformly from a low-dimensional manifold $\Mc \subseteq \R^d$. We construct a proximity graph $G$ of $\Xc$ that allows us to give non-asymptotic error bounds for for the approximation of $\M$'s curvature from data. Moreover, we show that if $\Mc$ has curvature bounded below by a positive constant, then so does $G$ with high probability. To the best of our knowledge, our local consistency result presents the first \emph{non-asymptotic} guarantees of this kind. In addition, we believe that our work provides the first consistency results for global curvature bounds. We complement our theoretical investigation of continuum limits with a discussion of potential applications to manifold learning.

We conclude with a brief discussion of avenues for future investigation. A limitation of the present work is the assumption that $\Xc$ is a uniform sample. Future work may investigate whether it is possible to adapt these results to other data distributions. Furthermore, we have assumed that the sample is noise-free; it would be interesting to analyze the noisy case with different noise models. In addition, one setting investigated in this work implicitly assumes access to a sufficiently good data-driven estimator for the geodesic distance. While we have suggested some directions for constructing such estimator, we believe that this question is of interest in its own right and deserves more attention. We would also like to highlight the ``shrinking" factor $s_K$ that appears in our main Theorems \ref{thm:GlobalBounds1} and \ref{thm:GlobalBounds2} should be removable with a much more detailed analysis.
Further applications of the global curvature lower bounds may arise in the study of Langevin dynamics on manifolds, specifically when utilizing graph-based constructions to define suitable discretizations of the infinitesimal generators of the stochastic dynamics of interest.

 \section*{Acknowledgements}
 The authors would like to thank Prasad Tetali for enlightening discussions and for providing relevant references. This work was started while the authors were visiting the Simons Institute to participate in the program ``Geometric Methods in Optimization and Sampling" during the Fall of 2021. The authors would like to thank the organizers of this program and the Simons Institute for support and hospitality. During the visit, MW was supported by a Simons-Berkeley Research Fellowship. NGT was supported by NSF-DMS grant 2005797 and would also like to thank the IFDS at UW-Madison and NSF through TRIPODS grant 2023239 for their support.

\bibliography{references}

\appendix

\section{Derivation of \eqref{eq:HigherOrderApprox}}
\label{sec:EuclideanVsGeodesic}
 

Let $\gamma: [0, \infty) \rightarrow \M \subseteq \R^d$ be a unit speed geodesic in $\M$. We will assume without the loss of generality that $x=\gamma(0) = 0$. 
At least for small enough time $t \leq t_0$, we have:
\[ d_\M(x, \gamma(t) )  = t .\]
We now compare the above with
\[ |x - \gamma(t)| = |\gamma(t)|,  \]
the Euclidean distance between $x$ and $\gamma(t)$. For that purpose we define the function 
\[ h(t) := t^2 - |\gamma(t)|^2, \quad t \in [0,t_0].  \]
A direct computation using the fact that $\langle \dot{\gamma}(t) , \dot{\gamma}(t) \rangle=1$ reveals the following expressions for the first four derivatives of $h$: 
\[ h'(t) = 2t - 2 \langle \dot{\gamma}(t),  \gamma(t) \rangle, \]
 \[ h''(t) = - 2 \langle \ddot{\gamma} (t) , \gamma(t) \rangle,  \]
 \[  h'''(t) = - 2 \langle \dddot{\gamma} (t) , \gamma(t) \rangle - 2 \langle \ddot{\gamma} (t) , \dot {\gamma}(t) \rangle = -  2 \langle \dddot{\gamma} (t) , \gamma(t) \rangle -  \frac{d}{dt} \langle \dot{\gamma} (t) , \dot {\gamma}(t) \rangle =  -  2 \langle \dddot{\gamma} (t) , \gamma(t) \rangle,   \]
 \[ h''''(t) = - 2 \langle \ddddot{\gamma} (t) , \gamma(t) \rangle - 2 \langle \dddot{\gamma} (t) , \dot {\gamma}(t) \rangle =  - 2 \langle \ddddot{\gamma} (t) , \gamma(t) \rangle + 2 \langle \ddot{\gamma} (t) , \ddot {\gamma}(t) \rangle.  \]
 In the above, the last expression for the fourth derivative follows from the following computation:
 \[ 0= \frac{d^2}{dt^2} \langle \dot{\gamma}(t) , \dot{\gamma}(t)\rangle = 2 \frac{d}{dt} \langle \ddot{\gamma}(t),  \dot{\gamma}(t) \rangle = 2 \langle \dddot{\gamma}(t), \dot{\gamma}(t) \rangle + 2 \langle \ddot{\gamma}(t), \ddot{\gamma}(t) \rangle. \]
 Now, at $t=0$ we have:
\[ h(0)=0, \quad h'(0)=0, \quad h''(0)=0, \quad h'''(0)=0, \quad h''''(0)= 2 \langle \ddot{\gamma}(0) , \ddot{\gamma}(0) \rangle   ,\]
since we have assumed $\gamma(0)=0$. A Taylor expansion then shows that
\[ h(t) = \frac{1}{12} \langle \ddot{\gamma}(0) , \ddot{\gamma}(0) \rangle t^4 + O(t^5).  \]
Hence
\begin{align*}
 t &= |\gamma(t)| \sqrt{ 1 + \frac{1}{12} \langle \ddot{\gamma}(0) , \ddot{\gamma}(0)\rangle \frac{ t^4}{|\gamma(t)|^2} + \frac{1}{|\gamma(t)|^2}O(t^5)   }
 \\ & = |\gamma(t)| \left(  1 + \frac{1}{24} \langle \ddot{\gamma}(0) , \ddot{\gamma}(0)\rangle \frac{ t^4}{|\gamma(t)|^2} + \frac{1}{|\gamma(t)|^2}O(t^5)      \right)
 \\&=  |\gamma(t)| + \frac{1}{24} \langle \ddot{\gamma}(0) , \ddot{\gamma}(0)\rangle t^3 + O(t^4).  
\end{align*}

 \section{Quantitative estimates of second fundamental form from data}
\label{sec:EstimatesSecondFundForm}
 
 
 Here we review some related literature on estimating the second fundamental form of a manifold $\M$ embedded in $\R^d$ with data. Kim et al.~\citep{kim_curvature-aware_2013} propose an estimator for the second fundamental form for embedded submanifolds, which is the class of manifolds considered in this paper. Specifically, they suggest to construct an estimator of the Hessian $\restr{H_h}{x}$ of the defining function $h$ at each point $x \in \Xc$ (recall Definition \ref{def:Embedded}). To do this, they fit a quadratic polynomial $p_h$ to the defining function in a small neighborhood of $x$ and assume $\restr{H_{p_h}}{x} \approx \restr{H_h}{x}$. They show that such an approximation convergences indeed asymptotically to the second fundamental form:
 \begin{theorem}[\citep{kim_curvature-aware_2013}]
     Let the coefficients $\tilde{A}_x$ of the polynomial $p_h$ be determined by solving the weighted least-squares problem
     \begin{equation*}
         \tilde{A}_x = {\rm argmin}_{Q} \| K_x(XQ-h) \| \approx A_x \; ,
     \end{equation*}
     where $X$ is the matrix of second-order monomials of points in $\Xc$ centered at $x$, 
     \begin{align*}
         A_x &=\frac{1}{2} \Big[ \big[\restr{H_h}{x}\big]_{1,1},
         \big[\restr{H_h}{x}\big]_{1,2},\dots,
         \big[\restr{H_h}{x}\big]_{d,d}         \Big]^T \\
         \tilde{A}_x &=\frac{1}{2}\Big[ \big[\restr{H_{p_h}}{x}\big]_{1,1},
         \big[\restr{H_{p_h}}{x}\big]_{1,2},\dots,
         \big[\restr{H_{p_h}}{x}\big]_{d,d}         \Big]^T
         \; ,
     \end{align*}
     and $K_x$ a diagonal matrix with ${\rm diag}(K_x)=\mathds{1}_{\| x_i - x \| \leq \veps}$. Then $\| A_x - \tilde{A}_x \| \rightarrow 0$ for all $x \in \Xc$ as $n \rightarrow \infty, \veps \rightarrow 0$.
 \end{theorem}
 \noindent A proof can be found in~\citep[Appendix, sections 1-2]{kim_curvature-aware_2013}.\\
 
 While this result holds for any manifold considered in this paper, it provides only \emph{asymptotic} guarantees. 
 Aamari and Levrard~\citep{aamari_nonasymptotic_2019} show that under additional smoothness assumptions on the underlying manifold, one can indeed obtain non-asymptotic error estimates. They give minimax bounds of order $O \left( n^{\frac{2-k}{m}}\right)$~\citep[Theorems 4 and 5]{aamari_nonasymptotic_2019}
 for an estimator of the second fundamental form of a $C^k$-smooth embedded submanifold. 
Below, we briefly recall the minimax bounds for later reference. To state the results, we introduce the following additional notation. We define, as usual, the \emph{operator norm} of a linear map $T$ on $S \subset \mathbb{R}^d$ as $\| T \|_{op}:= \sup_{z \in S} \frac{\| Tz \|}{\| z \|}$. Let $\Mc$ be a $C^k$-manifold with $k \geq 3$ and reach $\tau \geq \tau_{\min} >0$. For $x \in \Xc$, one can define the local estimator
 \begin{align}\label{eq:poly-estimate}
  (\hat{\pi}_j,\hat{T}_{2,j}, \dots,\hat{T}_{k-1,j}) \in \operatornamewithlimits{\argmin}_{\pi, \operatornamewithlimits{\sup}_{2<i<k} \| T_i \|_{op}\leq 1} P_{n-1}^{(j)} \Biggl[ 
    \Big\| x - \pi(x) - \sum_{i=2}^{k-1} T_i (\pi(x)^{\otimes i}) \Big\|^2 \mathds{1}_{B(0,h)}(x)
  \Biggr] \; ,
 \end{align}
 where $\pi$ is an orthogonal projection onto $d$-dimensional subspaces and $T_i: (\mathbb{R}^m)^i \rightarrow \mathbb{R}^m$ a bounded symmetric tensor of order $i$. Moreover, $P_{n-1}^{(j)}$ denotes integration with respect to $\frac{1}{(n-1)} \sum_{i \neq j} \delta_{(x_i - x_j)}$, $z^{\otimes i}$ the $(m \times i)$-dimensional vector $(z, \dots, z)$ and $h\leq \frac{\tau_{\text{min}}}{8}$.  Aamari and Levrard give the following guarantee for a solution of Eq.~\ref{eq:poly-estimate} (adapted to our notation and assumptions):
 \begin{theorem}[\citep{aamari_nonasymptotic_2019}]\label{thm:amari}
Let $\Mc$ be a $C^k$-manifold with $k \geq 3$ and reach $\tau \geq \tau_{\min} >0$. For sufficiently large $n$, we have the following bounds:
     \begin{enumerate}
         \item Upper bound:
    \begin{equation}
        \mathbb{E} \Big[  \max_{1 \leq j \leq n} \| \mathrm{I\!I}_{x_jy} \circ \pi_{T_{x_j}\Mc}  - \hat{T}_{2,j} \circ \hat{\pi}_j  \|_{op}\Big] \leq C \left( \frac{\log n}{n-1} \right)^{\frac{k-2}{d}}
    \end{equation}
         \item Lower bound:
    \begin{equation}
       \inf_{1 \leq j \leq n} \; \mathbb{E} \Big[ \| \mathrm{I\!I}_{x_jy} \circ \pi_{T_{x_j}\Mc}  - \hat{T}_{2,j} \circ \hat{\pi}_j  \|_{op}\Big] \geq C' \left( \frac{1}{n-1}\right)^{\frac{k-2}{d}}
    \end{equation}
     \end{enumerate}
     Here, $C,C'$ are constants depending on $d,k,\tau_{\min}$; $n$ is assumed to be sufficiently large, such that $C^{-1} \geq \big( \sup_{2 \leq i \leq k}  \| T_i^* \|_{op}  \big)^{-1}$ for the estimator.
 \end{theorem}
The geometric optimization problem Eq.~\eqref{eq:poly-estimate} can be viewed as an optimization task on the Grassmannian manifold~\citep{usevich}. 
While it may seem challenging to solve at first, it was shown in~\citep{divol} that it has a locally geodesically convex formulation, which can be efficiently solved via standard first-order solvers, such as Riemannian gradient descent. 
We further note that similar estimation results were also obtained in related work by Cao et al.~\citep{cao2021efficient}.

Notice that Theorem~\ref{thm:amari} provides \emph{error bounds} in expectation, which do not immediately translate into concentration bounds. While a development of such concentration bounds is beyond the scope of the present paper, we briefly comment on a possible avenue for deriving them, at least in the large-sample regime. Specifically, given a sufficiently large sample, one may construct approximations of tangent spaces via principal component analysis. Developing a means to track the change in the tangent space as we move along the manifold could deliver an approximation of the second fundamental form, which, in turn, would allow for deriving concentration bounds.


\end{document}